\documentclass[oneside]{amsart}
\usepackage[top=1in, bottom=1in]{geometry}
\usepackage[colorlinks = true, citecolor=blue]{hyperref}
\usepackage[utf8]{inputenc}

\usepackage{multirow}
\usepackage{booktabs}
\usepackage[T1]{fontenc}

\usepackage{amsmath}
\usepackage{graphicx,subfigure}
\usepackage[mathscr]{eucal}

\usepackage{enumitem}
\usepackage{ifthen}
\usepackage{cases}
\usepackage{amssymb}[psamsfonts]
\usepackage{amsfonts,bbm}
\newcommand{\Z}{\mathbb Z}
\newcommand{\R}{\mathbb R}

\newcommand{\positR}{(0,\infty)}

\newcommand{\C}{\mathbb C}
\newcommand{\N}{\mathbb N}
\newcommand{\E}{\mathbb E}

\newcommand{\im}{\mathbf{i}}

\newcommand{\K}{\mathcal K}
\renewcommand{\L}{\mathcal L}
\renewcommand{\P}{\mathbb P}

\newcommand{\Yb}{\mathbf{Y}}

\DeclareMathOperator*{\spec}{Spec}

\newcommand{\Db}{\mathbf D}
\newcommand{\Bb}{\mathbf B}

\newcommand{\Pc}{\mathcal P}
\newcommand{\Qb}{\mathbf Q}
\newcommand{\Xb}{\mathbf X}

\newcommand{\Cb}{\mathbf C}
\newcommand{\hRb}{\hat{\mathbf R}}
\newcommand{\hUpsilon}{\hat\Upsilon}
\newcommand{\hf}{\hat f}
\newcommand{\Rb}{\mathbf R}
\newcommand{\Zb}{\mathbf Z}
\newcommand{\Sb}{\mathbf S}

\newcommand{\Ib}{\mathbf I}
\newcommand{\Ab}{\mathbf A}

\newcommand{\bSb}{\underline{\mathbf S}}
\newcommand{\gb}{\mathbf{g}}

\newcommand{\mb}{\mathbf{m}}
\newcommand{\nb}{\mathbf{n}}
\newcommand{\ub}{\mathbf{u}}
\newcommand{\vb}{\mathbf{v}}
\newcommand{\xb}{\mathbf{x}}
\newcommand{\yb}{\mathbf{y}}
\newcommand{\zb}{\mathbf{z}}

\newcommand{\ind}{\mathbbm{1}}
\newcommand{\eqinlaw}{\overset{\L}{=}}
\newcommand{\lmax}{\lambda_{\mathrm{max}}}
\newcommand{\lmin}{\lambda_{\mathrm{min}}}

\newcommand{\gNc}{\mathcal{CN}}
\newcommand{\gNr}{\mathcal N}
\newcommand{\gN}{\mathcal N}
\newcommand{\dd}{\,\mathrm d}
\newcommand{\tq}{\,:\,}
\newcommand{\as}{\mathrm{a.s}}
\DeclareMathOperator*{\tr}{tr}

\DeclareMathOperator*{\diag}{diag}

\DeclareMathOperator*{\cov}{Cov}

\DeclareMathOperator*{\esssup}{ess\,sup}
\DeclareMathOperator*{\essinf}{ess\,inf}

\newcommand{\tran}{\top}

\newcommand{\cvweak}[1][]{%
\ifthenelse{\equal{#1}{}}{\xrightarrow{\mathcal D}}{\xrightarrow[#1]{\mathcal D}}%
}

\usepackage{amsthm}
\newtheorem{Th}{Theorem}[section]
\newtheorem{lemma}[Th]{Lemma}
\newtheorem{prop}[Th]{Proposition}
\newtheorem{corol}[Th]{Corollary}

\theoremstyle{definition}

\newtheorem{defi-prop}[Th]{Definition-Proposition}

\newtheorem*{example*}{Example}
\theoremstyle{remark}
\newtheorem*{Rq*}{Remark}
\newtheorem{Rq}{Remark}[section]

\numberwithin{equation}{section}
\newboolean{supplementary}
\setboolean{supplementary}{false}

\title[Ratio-consistent estimation for LRD Toeplitz covariance]{Ratio-consistent estimation for long range dependent Toeplitz covariance with application to matrix data whitening}
\date{\today}

\author[P. Tian]{Peng Tian}
\address{Peng Tian, Department of Civil Engineering, The University of Hong Kong, Pokfulam Road, Hong Kong.}
\email{tianpeng83@gmail.com}

\author[J.F. Yao]{Jianfeng Yao}
\address{Jianfeng Yao, School of Data Science, The Chinese University of Hong Kong (Shenzhen)}
\email{jeffyao@cuhk.edu.cn}

\thanks{The authors gratefully acknowledge the support by Department of Statics and Actuarial Science, the University of Hong Kong.}


\begin{document}

\maketitle

\begin{abstract}
We consider a data matrix $\Xb:=\Cb_N^{1/2}\Zb\Rb_M^{1/2}$ from a multivariate stationary process with a separable covariance function, where $\Cb_N$ is a $N\times N$ positive semi-definite matrix, $\Zb$ a $N\times M$ random matrix of uncorrelated standardized white noise, and $\Rb_M$ a $M\times M$ Toeplitz matrix. Under the assumption of long range dependence (LRD), we re-examine the consistency of two toeplitzifized estimators $\hRb_M$ (unbiased) and $\hRb_M^b$ (biased) for $\Rb_M$, which are known to be norm consistent with $\Rb_M$ when the process is short range dependent (SRD). However in the LRD case, some simulations suggest that the norm consistency does not hold in general for both estimators. Instead, a weaker {\it ratio consistency} is established for the unbiased estimator $\hRb_M$, and a further weaker {\it ratio LSD consistency} is established for the biased estimator $\hRb_M^b$. The main result leads to a consistent whitening procedure on the original data matrix $\Xb$, which is further applied to two real world questions, one is a signal detection problem, and the other is PCA on the space covariance $\Cb_N$ to achieve a noise reduction and data compression.
\end{abstract}

\section{Introduction}
Consider a random data matrix of the form
\begin{equation}
    \Xb=\Cb_N^{1/2}\Zb\Rb_M^{1/2} \label{eq:def_X}
\end{equation}
where $\Cb_N=(C_{N,i,j})$ and $\Rb_M=(R_{M,i,j})$ are $N\times N$ and $M\times M$ positive semi-definite Hermitian matrices, respectively, and $\Zb$ is a white noise array of size $N\times M$. The matrix $\Xb=(X_{i,j})$ has a so-called separable covariance function, that is, $\cov(X_{i,j}, X_{i',j'}) = C_{N,i,i'} R_{M,j,j'}$. In other words, $\Cb_N$ and $\Rb_M$ represent the covariance function between rows and columns of $\Xb$, respectively. In this paper, by mimicking a stationary time series structure across the column vectors, we assume that $R_M=(r_{i-j})$ is a Toeplitz covariance matrix. Such data matrices appear in many applications, for example, as the noise part of a signal-plus-noise models in signal processing problems \cite{vinogradova2015estimation, couillet2015rmt, terreaux2017robust}, or portfolio optimization problems \cite{jay2020improving}, or as a stand-alone model in \cite{vallet2017performance}. We note that if $\Cb_N$ is identity, the rows of $\Xb$ can be interpreted as i.i.d copies of a section of stationary process considered in \cite{merlevede2016empirical, merlevede2019unbounded, tian2020tracy}, and if $\Cb_N$ is diagonal with i.i.d random entries independent of $\Zb$, the rows of $\Xb$ can be interpreted as some i.i.d elliptically symmetric random vectors with Toeplitz scatter matrix $\Rb_M$,  which were considered in e.g. \cite{terreaux2017robust, jay2020improving}.

In the large dimensional context where both $M$ and $N$ are large, the estimation of $\Rb_M$ or $\Cb_N$ from the observable data $\Xb$ is a challenging question. Consider the following "sample covariance matrices"
\begin{equation}
    \Sb_X:=\frac{1}{M}\Xb\Xb^*, \quad \bSb_X:=\frac{1}{N}\Xb^*\Xb. \label{eq:intro:sample_covariance}
\end{equation}
Then direct calculation shows that $\E(\Sb_X)=(M^{-1}\tr\Rb_M)\Cb_N$ and $\E(\bSb_X)=(N^{-1}\tr\Cb_N)\Rb_M$. In other words, $\Sb_X$ and $\bSb_X$ are unbiased (up to scalar factors) estimators of the matrices $\Cb_N$ and $\Rb_M$, respectively. However, the results from Random matrix theory (RMT) show that neither of them is consistent in the large dimensional regime, see e.g. \cite{marvcenko1967distribution, silverstein1995strong, lixin2006spectral}.

Taking into account the Toeplitz structure of $\Rb_M$, it is possible to construct better estimators for $\Rb_M$. In \cite{vinogradova2015estimation}, the authors considered two estimators
\begin{equation}
    \hRb_M:=(\hat r_{i-j})_{1\le i,j\le M}, \quad \hRb_M^b:=(\hat r_{i-j}^b)_{1\le i,j\le M}, \label{eq:intro:S_toeplitzified}
\end{equation}
where
\begin{equation}
    \hat r_k:=\frac{1}{M-|k|}\sum S_{i+k,i}\ind_{\{1\le i,i+k\le M\}}, \quad \hat r_k^b:=\frac{1}{M}\sum S_{i+k,i}\ind_{\{1\le i,i+k\le M\}} \label{eq:intro:toepl_averaging}
\end{equation}
with $S_{i,j}$ the entries of $\bSb_X$. Let $\xi_N=\tr\Cb_N/N$. Note that $\hRb_M$ is an unbiased estimator of $\xi_N\Rb_M$, whereas $\hRb_M^b$ is biased. It is proved in \cite{vinogradova2015estimation, couillet2015rmt} that, if the entries of $\Zb$ are i.i.d standard complex Gaussian and if the sequence $(r_k)_{k\in\Z}$ in $\Rb_M$ is absolutely summable, both estimators are (spectral) norm consistent, that is,
\begin{equation}
    \left\|\hRb_M-\xi_N\Rb_M\right\| \xrightarrow{\mathrm{a.s}} 0\,, \quad \left\|\hRb_M^b-\xi_N\Rb_M\right\| \xrightarrow{\mathrm{a.s}} 0 \label{eq:converge_strong}
\end{equation}
as $M,N\to\infty$ with $N/M\to c\in\positR$. Such consistent estimators can be used to whiten the correlation between the columns of $\Xb$ in order to facilitate the inference on $\Cb_N$, as done in \cite{vinogradova2015estimation} and other articles. 

The absolute summability of the sequence $(r_k)_{k\in\Z}$ means that the columns of $\Xb$ are short range dependent (SRD), and this is crucial to the norm consistency \eqref{eq:converge_strong}. As an opposite scenario, the phenomenon of long range dependence (LRD) has been frequently observed in various fields like engineering and economic processes (see \cite{giraitis2012large, sheng2011fractional, doukhan2002theory, rangarajan2003processes} and the references therein). In this paper, we parameterize the LRD by $a\in(0,1)$, and study the consistency properties of $\hRb_M$ and $\hRb_M^b$ in the LRD case. Our main results are
\begin{enumerate}[label={(\roman{enumi})}]
    \item The norm consistency \eqref{eq:converge_strong} for the unbiased estimator $\hRb_M$ continues to hold in the LRD case for the lower half  $a\in(0,1/2)$ (the pro-SRD side), but simulation studies suggest that it may not hold for the upper  half  $a\in(1/2,1)$, see \S\ref{subsec:simu_norm_inconsistency}. \label{enum:The-norm-consistency}
    \item The unbiased estimator $\hRb_M$ is ratio consistent in the sense that \label{enum:The-unbiased-estimator}
    \begin{equation}
    \left\|\Rb _M^{-1/2}\hRb_M\Rb _M^{-1/2}-\xi_N\Ib\right\|\xrightarrow{\as} 0, 
    \label{eq:intro:ratio_consist}
    \end{equation}
    and consequently, 
    \begin{equation}
    \left\|\Rb _M^{-1/2}\hRb_M^{1/2}-\sqrt{\xi_N}\Ib\right\|\xrightarrow{\as} 0;
    \label{eq:intro:ratio_consist_direct}
    \end{equation}
    \item The biased estimator $\hRb_M^b$ is not ratio consistent. We will prove that under the same conditions with $\xi_N=1$, we have almost surely
    \begin{equation}
       \|\Rb_M^{-1/2}\hRb_M^b\Rb _M^{-1/2}- \Ib\|\not\to 0. \label{eq:intro:ratio_inconsistency_biased}
    \end{equation}
    As a corollary, $\hRb_M^b$ is not norm consistent, that is, almost surely, \begin{equation}
       \|\hRb_M^b-\Rb _M\|\not\to 0. \label{eq:intro:norm_inconsistency_biased}
    \end{equation}
    \label{enum:The-biased-estimator}
    \item \label{enum:A-weaker-ratio_2}A weaker ratio LSD consistency holds for $\hRb_M^b$ in the sense that the empirical spectral distribution (ESD) of $\Rb_M^{-1/2}\hRb_M^b\Rb_M^{-1/2}$ converges to $\delta_1$ (assuming $\xi_N\to 1$).
\end{enumerate}

An immediate application of \ref{enum:The-unbiased-estimator} is to whiten the correlation $\Rb_M^{1/2}$ in the data $\Xb$ by multiplication of $\hRb_M^{-1/2}$. Let
$$\Yb_w:=\Xb\hRb_M^{-1/2},\quad\text{and}\quad \Yb_{Rid}:=\xi_N^{-1/2}\Cb_N^{1/2}\Zb,$$
then by \eqref{eq:intro:ratio_consist_direct}, under some proper conditions, we will have 
$$\|\Yb_w-\Yb_{Rid}\|\xrightarrow{\as}0.$$
Note that the matrix $\Yb_{Rid}$ is not observable but can be approximated by the whitened data matrix $\Yb_w$, with which the inference on $\Cb_N$ becomes much easier, knowing that RMT contains many inference methods on $\Cb_N$ through the sample covariance matrix
$$\Sb_{Rid}:=\frac{1}{M}\Yb_{Rid}\Yb_{Rid}^*,$$
see e.g. \cite{yao2015sample} and \cite{johnstone2018pca}. In this paper, we apply this result in a signal plus noise model to detect signals, reduce noise and compress the data using PCA. Some other applications based on the ratio consistency of $\hRb_M$, such as the prediction of a multivariate separable time series, may be discussed in future works.

The result \ref{enum:The-biased-estimator} is striking because in the SRD case, $\hRb_M^b$ is a better estimator for $\Rb_M$ than $\hRb_M$ with a smaller variance. In the LRD case however, the bias is no longer negligible. In \S\ref{subsec:ratio_inconsistency} we will illustrate by numeric simulations that this inconsistency may invalidate some subsequent whitening procedures. However, $\hRb_M^b$ has a weaker ratio LSD consistency \ref{enum:A-weaker-ratio_2} which may still be useful.

It is worthy of noticing that the entries $\hat r_k$ or $\hat r_k^b$, or some other banded/tapered estimators based on them are often used to estimate $r_k$ or the periodogram. Many results have been established in this aspect, see e.g. \cite{dai2004asymptotics,woodroofe1967maximum, liu2010asymptotics, wu2010covariances}. For matrix estimators, Ing et al. \cite{ing2016estimation} established the norm consistency of an estimator for $\Rb_M^{-1}$ under some special LSD conditions. However, to the authors' best knowledge, the ratio consistency of matrix estimators for the purpose of whitening in the presence of LRD has not yet been established before.

We now describe some important technical innovations introduced in this paper as compared to the existing literature for the SRD case. The general structure of the proof of our main theorem, Theorem~\ref{th:main_th} follows \cite{vinogradova2015estimation}. But unlike the reference where the white noise matrix $\Zb$ has i.i.d complex Gaussian entries, we also allow the rows $\zb_i$ of $\Zb$ to be uniformly distributed on a centered sphere in $\R^M$ or $\C^M$. This setting permits the rows of $\Xb$ to have more general elliptical distributions. Thus the $M$ columns of $\Zb$ are uncorrelated but dependent and new tools are needed in various moment estimation involving these noise variables such as
$$\E e^{\tau\sum_m\sigma_m|Z_{1,m}|^2}$$
where $Z_{1,m}, m=1,\dots,M$ are the elements of the first row.

In the proof of Theorem~\ref{th:main_th}, an accurate upper bound for a trace of non commutative product of several Toeplitz matrices (of type $\tr(\Db_M(\theta)\Rb_M\Db_M(\theta)^*\Bb_M)^2$, where $\Rb_M, \Bb_M$ are Toeplitz, and $\Db_M(\theta)$ is diagonal depending on $\theta\in(0,2\pi)$ ) is crucial. In \cite{vinogradova2015estimation}, a global bound was obtained using linear algebraic method thanks to the boundedness of $\|\Rb_M\|$. In this paper, since $\|\Rb_M\|\to \infty$, the global bound is not sufficient. We rewrite the same trace as a two dimensional harmonic integral and then estimate its bound in terms of $M$ and the underlying spectral density $f(\theta)$. This is not trivial since $f$ has a power singularity at $0$. 

\subsection*{Notations.} Matrices are denoted by bold capital characters, row or column vectors are denoted by bold characters. For $x\in\R$, $\delta_x$ denotes the Dirac measure at $x$. For a Hermitian $N\times N$ matrix $\Sb$, its eigenvalues are denoted as $\lambda_1(\Sb)\ge \dots \ge \lambda_N(\Sb)$, and $\mu^{\Sb}:=N^{-1}\sum_{i=1}^N\delta_{\lambda_i(\Sb)}$ denotes the ESD of $\Sb$. The largest and smallest eigenvalues of $\Sb$ are also denoted by $\lmax(\Sb)$ and $\lmin(\Sb)$, respectively. For a matrix $\Ab$, $A_{i,j}$ stands for its $i$th row and $j$th column element, $\|\Ab\|$ and $\|\Ab\|_F$ its operator norm and Frobenius norm, respectively, and $\Ab^*$ its conjugate transpose. The spectrum of a square matrix $\Ab$ is denoted by $\spec(\Ab)$. For a function $f$, $\|f\|_1$ and $\|f\|_\infty$ stand for its $L^1$ and $L^\infty$ norm, respectively. If $a,b$ are two elements of a Hilbert space, we denote their inner product as $\langle a, b\rangle$. The symbol $K$ denotes a constant which may take different values from one place to another. If several constants are needed in one expression, we will denote them by $K_1,K_2,\dots$. For two sequences of positive numbers $a_n$ and $b_n$, $a_n \lesssim (\gtrsim) b_n$ means that there exists a constant $K>0$ such that $a_n\le (\ge) Kb_n$ for all $n$, and $a_n \asymp b_n$ means that there exist constants $0<K_1<K_2$ such that $K_1b_n \le a_n \le K_2 b_n$ for all $n$. The underlying constants may depend on the spectral density $f$ defined in the assumption \ref{ass:spec_den_behav_0} below, but do not depend on any other variables in this article. The notation $a_n\gg b_n$ (resp. $a_n\ll b_n$) means that $a_n/b_n\to\infty$ (resp. $a_n/b_n\to 0$).

\subsection*{Organization}
In \S\ref{sec:model_main_results}, we state the main results. In \S\ref{sec:applications}, we develop applications to signal detection and high-dimensional PCA for the data matrix $\Xb$. In \S\ref{sec:num_illustrations}, we provide in \ref{subsec:simu_norm_inconsistency} numeric simulations to show the norm inconsistency of $\hRb_M$, and to illustrate of the ratio inconsistency of $\hRb_M^b$ with its impact to the whitening procedure in \ref{subsec:ratio_inconsistency}. We prove our main results Theorem~\ref{th:main_th} and Proposition~\ref{prop:inconsis_Rb} in \S\ref{sec:proof_main_th} and \S\ref{sec:proof_inconsis_Rb}, respectively. The other proofs are given in the appendices.

\section{Main results} \label{sec:model_main_results}
\subsection{Model and assumptions}
We consider the random data matrix $\Xb$ in \eqref{eq:def_X} with the following assumptions.
\begin{enumerate}[leftmargin=*, label={\bf A\arabic{enumi}}]
    \item \label{ass:Gaussian} \label{ass:1st} The rows $\zb_i$ of $\Zb$ are i.i.d real or complex random vectors, either standard normal distributed, or distributed as $\sqrt{M}\ub$ where $\ub$ follows the uniform (Haar) measure on the unit sphere in $\R^M$ or $\C^M$.
    \item \label{ass:C_moments} \label{ass:C_bound}\label{ass:C_diag} The matrices $\Cb_N$ are nonnegative and diagonal, i.e. $\Cb_N=\diag(c_1,\dots,c_N)$ where $c_n\ge 0$ and $c_n$ may also depend on $N$ for $n=1,\dots,N$. Moreover there exist constants $C>0$ and $\kappa>0$ such that
    $$\frac{1}{N}\tr\Cb_N^2 \le C,\quad \|\Cb_N\|\le \kappa\log M.$$ 
    \item \label{ass:spectral_density}\label{ass:1st_bias} The Toeplitz matrices $\Rb_M=(r_{i-j})_{i,j=1}^M$ have a positive spectral density $f\in L^1(-\pi,\pi)$ which is bounded in any set of the form $[-\pi, \pi]\backslash (-\delta,\delta)$ with $\delta>0$.
    \item \label{ass:spectral_bound_below} The spectral density $f$ is bounded away from $0$:
    $$\essinf_{\theta\in(-\pi,\pi)}f(\theta)>0.$$
    \item \label{ass:spec_den_behav_0} The spectral density $f$ is even and has the following asymptotic behavior near $0$:
    $$f(x)=\frac{L(|x|^{-1})}{|x|^a}$$
    for $x\in[-\pi,\pi]\backslash\{0\}$ where $a\in(0,1)$ and $L$ defined in $[\pi^{-1},\infty)$ is a slowly varying function at $\infty$. 
    \label{ass:last}
\end{enumerate}

The assumptions on $\Zb$ and $\Cb_N$ allow the matrix $\Xb$ to cover two types of models. When the matrix $\Zb$ has i.i.d standard real (resp. complex) Gaussian entries and $\Cb_N$ is a real symmetric (resp. complex Hermitian) deterministic matrix, each column of $\Xb$ is distributed as $\gN(0,\Cb_N)$ and the correlation between two columns $\xb_i,\xb_j$ is $\gamma(i-j)\Cb_N$, which is the product of a scalar $\gamma(i-j)$ depending only on the difference of their indices $i-j$ and a fixed matrix $\Cb_N$. Then $\Xb$ represents a $N$-dimensional stationary Gaussian process with a separable correlation structure. Note that in this case, the multivariate process $\Xb$ is a linear transform by $\Cb_N^{1/2}$ of $N$ i.i.d samples of a univariate stationary process. By the orthogonal (resp. unitary) invariance of the columns in $\Xb$, we can assume that $\Cb_N$ is diagonal without modifying the distribution of $\Sb_X$, $\hRb_M$ and $\hRb_M^b$ defined in \eqref{eq:intro:sample_covariance}-\eqref{eq:intro:toepl_averaging}. When the rows $\zb_i$ are distributed as $\sqrt{M}\ub$ where $\ub^\tran$ follows the uniform (Haar) measure on the unit sphere in $\R^M$ or $\C^M$, and $\Cb_N=\diag(\nu_1,\dots,\nu_N)$ with $\nu_i$ some i.i.d nonnegative random variables, independent of $\zb_i$, we can write the rows of $\Xb$ as
$$\xb_i=\sqrt{\nu_i}\zb_i\Rb_M^{1/2}.$$
Then $\xb_i^\tran$ has an elliptical distribution, and the data matrix $\Xb$ represents a set of i.i.d samples of elliptical random vectors with Toeplitz scatter matrix $\Rb_M$. Because $\Cb_N$ is independent of $\Zb$, we can treat $\Cb_N$ as deterministic by standard conditioning arguments, and our results are still applicable.

Recall that the spectral density of a sequence of Toeplitz matrices $\Rb_M=(r_{i-j})$ is a function $f\in L^1(-\pi,\pi)$ whose Fourier coefficients are $r_k$:
$$r_k=\frac{1}{2\pi}\int_{-\pi}^\pi f(x)e^{-\im k x}\dd x.$$
If $f$ is real, then $\Rb_M$ is Hermitian; if $f$ is positive, then $\Rb_M$ is positive definite; if $f$ is positive and even, then $\Rb_M$ is real symmetric and positive definite. We will consider $f$ as a $2\pi$-periodic function so that $f(x)$ is well defined by periodicity for all real $x$. Note that the assumption \ref{ass:spectral_bound_below} ensures that the smallest eigenvalue of $\Rb_M$ is positive and bounded away from $0$, thus $\Rb_M$ is invertible for all $M$ with $\|\Rb_M^{-1}\|$ bounded. If $\Rb_M$ is the autocovariance matrix of a stationary process and satisfies \ref{ass:spectral_density}, \ref{ass:spec_den_behav_0}, then the process is LRD by \cite[Definition~2.1.5 (Condition~IV)]{pipiras2017long}. 

Recall the definition of $\Sb_X$ in \eqref{eq:intro:sample_covariance}, the two estimators $\hRb_M$ and $\hRb^b_M$ in \eqref{eq:intro:S_toeplitzified}-\eqref{eq:intro:toepl_averaging}, and $\xi_N=N^{-1}\tr\Cb_N$. 

\subsection{Consistency properties of the unbiased estimator $\hRb_M$}
In this subsection we study the consistency of the unbiased estimator $\hRb_M$. We first point out that using a norm bound of $\Rb_M$ given in Lemma~\ref{lem:norm_toeplitz} below, and simply adapting the proof of \cite[Theorem~2]{vinogradova2015estimation}, one can prove the following large deviation bound.
\begin{prop}Assume that \ref{ass:Gaussian},\ref{ass:spectral_density} and \ref{ass:spec_den_behav_0} hold with $\Zb$ having i.i.d complex Gaussian entries, $\Cb_N$ bounded in spectral norm, and $0<a<1/2$. Also assume that $N/M\to c\in\positR$. Then there exists a constant $K>0$, such that for any fixed $x>0$, we have
\begin{equation}
    \P\left(\left\|\hRb_M-\xi_N\Rb_M\right\| >x \right) \le \exp\left(-\frac{KcM^{1-2a}x^2}{\|\Cb_N\|^2 L^2(M)\log M}(1+o(1))\right). \label{eq:large_dev_normsp}
\end{equation}
where $o(1)$ is with respect to $M$ and depends on $x$.
\end{prop}
\begin{proof}
By checking carefully the proof of \cite[Theorem~2]{vinogradova2015estimation}, we note that we can adapt it by replacing all the occurrences of the infinite-norm of the spectral density ($\|\boldsymbol{\Upsilon}\|_\infty$ in \cite{vinogradova2015estimation}) with the spectral norm of the matrix $\|\Rb_M\|$, {
and also re-analyzing the contribution of the term $\|\Rb_M\|$ where this term was previously bounded in \cite{vinogradova2015estimation}}. We also need to consider the contribution of $\Cb_N$, which will introduce the factor $\|\Cb_N\|^2$ in the denominator of the exponential bound there. Therefore, we have the following estimate:  
\begin{equation}
    \P\left(\left\|\hRb_M-\xi_N\Rb_M\right\| >x \right) \le \exp\left(-\frac{KNx^2}{\|\Cb_N\|^2\|\Rb_M\|^2\log M}(1+o(1))\right).
\end{equation}
Taking into account the bound $\|\Rb_M\|\asymp M^aL(M)$ given in Lemma~\ref{lem:norm_toeplitz} below, the result follows. The details are omitted.
\end{proof}
This proposition implies that in some LRD cases, where $\Rb_M$ satisfies \ref{ass:spec_den_behav_0} with $0<a<1/2$, the unbiased estimator $\hRb_M$ is still norm consistent. However, this result has the following defects. Firstly, the result does not cover the case $1/2\le a<1$ (In fact we conjecture that in this case the norm consistency does not hold. This will be supported by simulations with an heuristic argument in \S\ref{subsec:simu_norm_inconsistency}). However in some applications such as whitening, the convergence of the ratio $\Rb_M^{1/2}\hRb_M^{-1}\Rb_M^{1/2}$ to some scaled identity $\xi \Ib$ suffices. Secondly, even for the norm consistent case, we can see that the convergence rate ensured by this proposition is no better than $\sqrt{\log(M)}L(M)/M^{1/2-a}$, which depends on $a$ and gets worse and worse when $a$ approaches $1/2$. Meanwhile, since $\Rb_M^{-1}$ is merely bounded under \ref{ass:spectral_bound_below}, the proposition provides the same convergence rate for the ratio $\Rb_M^{-1/2}\hRb_M\Rb_M^{-1/2}$.

These facts motivate us to establish the following large deviation bound for the ratio. 
\begin{Th}\label{th:main_th} Under the assumptions \ref{ass:1st} - \ref{ass:last}, for any sequence $x_M$ satisfying $M^{-\gamma}\lesssim x_M\lesssim C/\kappa$ for some $\gamma>0$, there exist $K>0$, such that for large enough $M$ and any $N\ge 1$, we have
\begin{equation}
    \P\left(\left\|\Rb_M^{-1/2}\hRb_M\Rb_M^{-1/2}-\xi_N\Ib\right\| >x_M \right) \le 2\exp\left(- \frac{Nx_M^2}{CK\log^2 M}+\beta\log M\right). \label{eq:main_result_large_dev}
\end{equation}
where $\beta$ is any integer larger than $2+a+\gamma$.
\end{Th}

\begin{Rq} The inequality \eqref{eq:main_result_large_dev} literally holds for large enough $M$ and all $N\ge 1$. Meaningful asymptotic results can be obtained  by considering $M,N\to \infty$ with different regimes. We leave the freedom of specifying a precise exploding speed of $M,N$ to specific context of applications. 
For example, in order to get the almost sure convergence
\begin{equation}
    \left\|\Rb_M^{-1/2}\hRb_M\Rb_M^{-1/2}-\xi_N\Ib\right\|\xrightarrow[]{\as}0, \label{eq:main_result_as_conv}
\end{equation}
we should assume $M\to\infty$ and $N=N(M)\gg \log^3 M$.

As another example, if the target is some precise convergence rate, we will consider small  $x_M$ such as  $x_M=A(\log^\frac{3}{2} M)/\sqrt{N}$ with a large enough constant $A$.
Then from \eqref{eq:main_result_large_dev}, we have
\begin{equation}
    \P\left(\left\|\Rb_M^{-1/2}\hRb_M\Rb_M^{-1/2}-\xi_N\Ib\right\| >x_M \right) \le 2M^{-(\frac{A^2}{CK}-\beta)}. \label{eq:proba_ineq:cvrate}
\end{equation}
So if we take {
$N=N(M)\lesssim M^{2\gamma}$ for some $\gamma>0$ (which is required by the statement of Theorem~\ref{th:main_th}) and a large enough $A$ (to make the RHS of \eqref{eq:proba_ineq:cvrate} summable), almost surely, for large enough $M$, }
\begin{equation}\left\|\Rb_M^{-1/2}\hRb_M\Rb_M^{-1/2}-\xi_N\Ib\right\|\le \frac{A\log^{\frac{3}{2}} M}{\sqrt{N}}.
\end{equation}
Of course, if we want this bound to vanish, we need $N\gg \log^3M$.

From this theorem, in order to accurately estimate  the autocovariance matrix $\Rb_M$, we need that the dimension $N$ is also large enough. This is natural since in the Gaussian case, $N$ is the number of i.i.d. copies of a univariate stationary process (the rows of $\Zb\Rb_M^{1/2}$) we used to construct  $\Xb_M$ via the  linear transform $\Cb_N^{1/2}$.
\end{Rq}

As a corollary of Theorem~\ref{th:main_th}, when $N,M$ are of the same order, the matrix $\hRb_M^{1/2}\Rb_M^{-1/2}$ is equivalent to $\sqrt{\xi_N}\Ib$. This fact is useful when we want to rebuild the uncorrelated data $\Yb=\Cb^{1/2}_N\Zb$ by whitening $\Rb_M^{1/2}$.
\begin{corol}\label{corol:conv_1/2power} Under the same assumptions as in Theorem~\ref{th:main_th}, assume moreover that $\xi_N$ is lower bounded from $0$. Then as $M,N\to\infty$ with $N \asymp M$,  we have almost surely
\begin{equation}
    \left\|\hRb_M^{1/2}\Rb_M^{-1/2}-\sqrt{\xi_N}\Ib\right\|\to 0. \label{eq:conv_1/2power}
\end{equation}
\end{corol}

Using the ratio consistency of $\hRb_M$, we develop a consistent whitening procedure on $\Xb$, which is further applied to signal detection, noise reduction and data compression. See \S\ref{sec:applications}.

When $\Cb_N$ is random and independent of $\Zb$ such that \ref{ass:C_moments} is almost surely satisfied for large enough $M,N$, the almost sure convergence \eqref{eq:main_result_as_conv} also holds. For example, if $\Cb_N=\diag(\nu_1,\dots, \nu_N)$ with $(\nu_i)_{i\in\N}$ a sequence of i.i.d sub-exponential (in the sense that $\P(|\nu_1|>t)\le K_1e^{-t/K_2}$ with some $K_1,K_2>0$ for any $t>0$) positive random variables satisfying $\E|\nu_1|^2=1$, then
$$\frac{1}{N}\sum_{i=1}^N\nu_i^2\xrightarrow[N\to\infty]{}1,$$
and there exists $\kappa>0$ such that almost surely for large enough $N$,
$$\max_{1\le i\le N}\{\nu_i\}\le \kappa\log N.$$
In this case \eqref{eq:main_result_as_conv} holds as $M,N\to\infty$ with $N\asymp M$.

\subsection{Consistency properties of the biased estimator $\hRb_M^b$}
In the SRD case, the biased estimator $\hRb^b_M$ has several advantages over $\hRb_M$. Firstly, it is structurally positive semi-definite (see Lemma~3 of \cite{vinogradova2015estimation}). Secondly, it has smaller deviation from its expectation (the fluctuation rate is lower than $N^{-\alpha}$ for any $\alpha<1$, see \cite{couillet2015rmt}), because the inaccuracy of the elements near the top-right and bottom-left corners is more reduced in $\hRb_M^b$ than in $\hRb_M$. 

However in the LRD case, $\hRb^b_M$ is no longer consistent with $\Rb_M$, even in the sense of ratio consistency. In fact, $\hRb_M^b$ is ratio consistent with its expectation
$$\Rb_M^b:=\E\hRb_M^b=\left(\left(1-\frac{|i-j|}{M}\right)r_{i-j}\right)_{i,j=1}^M,$$
whose difference from $\Rb_M$ is no longer negligible in the LRD case. This is precisely established below. 

\begin{prop}\label{prop:inconsis_Rb} Let $\Xb$ be defined in \eqref{eq:def_X} with \ref{ass:1st} - \ref{ass:last} hold. Suppose that $\xi_N=N^{-1}\tr\Cb_N=1$. Then as $M\to\infty$ with $N=N(M) \gg \log^3 M$, almost surely, $\hRb^b_M$ is ratio consistent with $\Rb_M^b$:
\begin{equation}
    \|(\Rb_M^b)^{-1/2}\hRb^b_M(\Rb_M^b)^{-1/2}-\Ib\|\to 0, \label{eq:ratio_consist_hRbRb}
\end{equation}
but not with $\Rb_M$:
\begin{equation}
    \|\Rb_M^{-1/2}\hRb^b_M\Rb_M^{-1/2}-\Ib\|\not\to 0. \label{eq:inconsistency}
\end{equation}
\end{prop}
Note that the inconsistency \eqref{eq:inconsistency} is a special phenomenon caused by LRD, because in the SRD case, as long as \ref{ass:spectral_bound_below} holds, $\|\Rb_M\|$ and $\|\Rb_M^{-1}\|$ are both bounded, then the ``ratio consistency'' and the ``norm consistency'' are equivalent.

The proof of Proposition~\ref{prop:inconsis_Rb} also shows that the inconsistency \eqref{eq:inconsistency} caused by LRD affects not only the biased estimator $\hRb_M^b$, but more generally a large class of tapered estimators of $\Rb_M$. Analogous to $\hRb_M^b$, we often taper the estimates of $r_k$ for values of $k$ close to $M-1$ in order to reduce the inaccuracy of these estimates. But in the LRD case, such tapering often modifies the asymptotic behavior of the largest eigenvalue of the resulting estimator for $\Rb_M$, which in turn destroys its ratio consistency, see the proof of Proposition~\ref{prop:inconsis_Rb}.

Despite this ratio inconsistency of $\hRb_M^b$ with $\Rb_M$, we find that only a small part of the eigenvalues of $\Rb_M^{-1}\hRb_M^b$ deviate from $1$. In fact, we will establish the ratio LSD consistency between the two matrices.

\begin{prop}\label{prop:lsd_consistent} Let $\Xb$ be defined in \eqref{eq:def_X} with \ref{ass:1st}-\ref{ass:last} hold. Suppose that $\xi_N\to 1$ as $N\to\infty$. Then as $M\to\infty$ with $N=N(M)\gg \log^3 M$, almost surely the ESD of $\Rb_M^{-1}\hRb_M^b$ converges weakly to $\delta_1$.
\end{prop}

Thanks to the ratio LSD consistency of $\hRb_M^b$ with $\Rb_M$, $\hRb_M^b$ can still serve as a good approximation for $\Rb_M$ in certain circumstances, see \S\ref{subsec:ratio_inconsistency}.

We end this section by recapitulating the consistency properties of estimators $\hRb_M$ and $\hRb_M^b$ with $\Rb_M$, in the cases of SRD and LRD, respectively.

\begin{table}[ht]
\centering
\caption{Recapitulation of consistency properties}
\begin{tabular}{lll}
\toprule
 & SRD & LRD \\
 \hline
\multirow{4}{*}{$\hRb_M$} & \multirow{2}{*}{norm consistent} & norm consistent when $0<a<1/2$ \\
& & *norm inconsistent when $1/2<a<1$ \\
 & ratio consistent &  ratio consistent \\
 & ratio LSD consistent & ratio LSD consistent \\
\hline
\multirow{3}{*}{$\hRb_M^b$} & norm consistent & norm inconsistent  \\
    & ratio consistent &  ratio inconsistent \\
    & ratio LSD consistent & ratio LSD consistent \\
\bottomrule 
\multicolumn{3}{l}{* is only supported by numerical studies. See Section~\ref{subsec:simu_norm_inconsistency}.}
\end{tabular}
\label{table:recapitulation_inconsistency}
\end{table}

\section{Applications to matrix data whitening} \label{sec:applications}
Suppose that we have observed a data matrix $\Xb=\Cb_N^{1/2}\Zb\Rb_M^{1/2}$ with $M,N$ large but of the same order, and we want to detect the spike eigenvalues and the associated eigenvectors of $\Cb_N$. If $\Rb_M$ is identity, the data matrix becomes $\Yb=\Cb_N^{1/2}\Zb$, and from RMT results, we can find the spike eigenvalues (usually the extreme ones in applications) of 
$$\Sb_Y=\frac{1}{M}\Yb\Yb^*=\frac{1}{M}\Cb_N^{1/2}\Zb\Zb^*\Cb_N^{1/2},$$
and calculate the spike eigenvalues of $\Cb_N$ using the formula relating the spike eigenvalues of $\Cb_N$ and $\Sb_Y$, see \cite[Chapter~11]{yao2015sample} for more details. However, for general matrix $\Rb_M$, especially when the underlying stationary process is LRD, the above method fails because the information relevant to $\Cb_N$ is mixed with the covariance matrix $\Rb_M$. In fact from \cite[Corollary~2.1]{tian2022joint} we know that if \ref{ass:spectral_density},\ref{ass:spec_den_behav_0} hold and if in addition, the LSD of $\Cb_N$ weakly converges, then for any fixed $m\ge 1$, the $m$ largest eigenvalues of $\Sb_X=M^{-1}\Xb\Xb^*$ are asymptotically equivalent to
$$\frac{\tr\Cb_N}{M} \{ \lambda_1(\Rb_M), \dots, \lambda_m(\Rb_M)\},$$
as $N,M\to\infty$ and $N/M\to c\in\positR$. Thus, even if $\Cb_N$ is bounded, the largest eigenvalues of $\Sb_X$ tend to infinity following those of $\Rb_M$, and only the summary statistic $\tr\Cb_N$ appears in their first order limits. It is thus very difficult, if not impossible, to estimate the number and locations of the spiked eigenvalues, and the associated eigenvectors of $\Cb_N$. This breaks down any PCA on the original data based on $\Cb_N$. 

Using Theorem~\ref{th:main_th} and Corollary~\ref{corol:conv_1/2power}, we can whiten and remove the time correlation $\Rb_M^{1/2}$ from the data matrix $\Xb$ by multiplying it with $\hRb_M^{-1/2}$. Consider the whitened data matrix
$$\Yb_w:=\Xb\hRb_M^{-1/2}$$
and the whitened sample covariance matrix
\begin{equation}
    \Sb_w:=\frac{1}{M}\Yb_w\Yb_w^*=\frac{1}{M}\Cb_N^{1/2}\Zb(\Rb_M^{1/2}\hRb_M^{-1}\Rb_M^{1/2})\Zb^*\Cb_N^{1/2}. \label{eq:Sw_pca}
\end{equation}
Consider also the corresponding matrices with $\Rb_M$ equaling to identity and $\Cb_N$ normalized by $\xi_N$:
\begin{equation}\Yb_{Rid}:=\xi_N^{-1/2}\Yb,\quad \Sb_{Rid}:=\frac{1}{M}\Yb_{Rid}\Yb_{Rid}^*=\frac{1}{\xi_N}\bSb_Y.
\end{equation}
Our theory ensures that the impact of the covariance matrix $\Rb_M$ is properly removed from $\Xb$ so that the matrix $\Sb_w$ is close to $\Sb_{Rid}$ in spectral norm.
\begin{prop}\label{prop:whitening_consistence}
Let \ref{ass:1st}-\ref{ass:last} hold with $\Cb_N$ bounded in spectral norm and $\xi_N=\tr\Cb_N/N\ge \epsilon>0$. Then as $N,M\to\infty$ with $N/M\to c\in\positR$, we have
$$\left\|\Sb_w -\Sb_{Rid}\right\|\xrightarrow{\mathrm{a.s}} 0.$$
\end{prop}
The proposition ensures that many statistical methods for the standard covariance matrix $\Sb_{Rid}=\xi_N^{-1}\Sb_Y$ are applicable to the whitened matrix $\Sb_w$. In the following, we develop two statistical applications of this whitening procedure with the following deformed "signal plus noise" model. 

Let $\Yb=(\yb_1,\dots,\yb_M)$ with
\begin{equation}
    \yb_j=\Ab\mb_j + \sigma \nb_j, \label{eq:def_ybj_pca}
\end{equation}
where $\Ab$ is a $N\times p$ matrix with $p\ge 0$ a fixed integer, $\sigma>0$ and $\mb_j\sim \gNc(0,\Ib_p)$, $\nb_j\sim \gNc(0,\Ib_N)$ are standard complex Gaussian vectors, independent of each other and across $i$. The vectors $\yb_j$ represent antenna data, where $\Ab\mb_j$ represent the signal and $\sigma \nb_j$ the noise. A little calculation shows that the matrix $\Yb_{Rid}$ can be written in the form
\begin{equation}
    \Yb=\Cb_N^{1/2}\Zb, \label{eq:def_Yb_sig}
\end{equation}
where $\Cb_N = \cov(\yb_i)=\Ab \Ab^*+\sigma^2 \Ib_N$, and $\Zb$ is a $N\times M$ matrix with i.i.d standard complex Gaussian entries. Note that in this case $\Cb_N$ are unitarily similar to the diagonal matrix
\begin{equation}
    \sigma^2\diag(\alpha_1,\dots,\alpha_p,1,\dots,1), \label{eq:def_diag_C}
\end{equation}
where $\alpha_i=1+\frac{\lambda_i(\Ab\Ab^*)}{\sigma^2}=:1+\beta_i$ are signal strengths, and $\beta_i:=\frac{\lambda_i(\Ab\Ab^*)}{\sigma^2}$ are signal-to-noise ratios.

Now suppose that the data matrix $\Yb$ is "polluted" during its transmission which takes the form of a LRD time series, and only the matrix
\begin{equation}
    \Xb=\Yb\Rb_M^{1/2} \label{eq:def_X_PCA}
\end{equation}
is observable, where $\Rb_M$ is a Toeplitz matrix satisfying \ref{ass:spectral_density}, \ref{ass:spectral_bound_below}, \ref{ass:spec_den_behav_0}. For the ease of numerical simulations, we let the entries of $\Rb_M$ be
\begin{equation}
    r_{i-j}=\frac{1}{(1+|i-j|)^{1-a}}. \label{eq:def_r_ij_power_a}
\end{equation}
By \cite[Proposition~2.2.14]{pipiras2017long}, $\Rb_M$ satisfies \ref{ass:spectral_density} and \ref{ass:spec_den_behav_0}. By \cite[Theorem~1.5, Chapter~V]{zygmund2002trigonometric}, $\Rb_M$ also satisfies \ref{ass:spectral_bound_below}. Indeed, if the diagonal entry $r_0$ is large enough such that $(r_n)_{n\ge 0}$ is convex, the spectral density of $\Rb_M$ is nonnegative. The minimal value of such $r_0$ is $2^a-3^{a-1}<1$ for $0<a<1$. Thus when we take $r_0=1$, the spectral density $f$ is larger than $1-2^a+3^{a-1}>0$. 

In \S\ref{subsec:detect_number_signals}, we detect the number of spikes in $\Cb_N$, and estimate the signal strength $\alpha_i$. In \S\ref{subsec:PCA}, we proceed a PCA on the data matrix $\Xb$, or on the whitened data $\Yb_w$ to reduce the noise, and the obtained matrices are of rank $p$, realizing a
compression of the original data matrix $\Xb$.

\subsection{Detection of the number of signals and estimation of their strengths $\alpha_i$} \label{subsec:detect_number_signals}
As an immediate application of the asymptotic proximity between $\Sb_w$ and $\Sb_{Rid}$, we propose two algorithms, 1) to identify the number of spikes $p$ from $\Sb_w$, and 2) to estimate the spikes $\alpha_1,\dots,\alpha_p$.  We assume that the spikes $\alpha_1,\dots,\alpha_p$ are simple. The following proposition is the theoretical base of our algorithms. It is a corollary of Proposition~\ref{prop:whitening_consistence} and  \cite[Theorem~2.9, Theorem~11.3]{yao2015sample}.

\begin{prop}\label{prop:Sw_LSD_spike} Let the rows of $\Yb$ be defined in \eqref{eq:def_ybj_pca} and assume that the conditions of Proposition~\ref{prop:whitening_consistence} hold. Then the LSD of $\Sb_w$ is the Mar\v cenko-Pastur distribution
\begin{equation}
    \mathbb{P}_{MP}(\dd\lambda):=\frac{\sqrt{[(\lambda^+-\lambda)(\lambda-\lambda^-)]_+}}{2\pi c\lambda}\ind_{\lambda\in[\lambda^-,\lambda^+]}\dd \lambda +\left(1-c^{-1}\right)\delta_0(\dd\lambda)\ind_{\{c>1\}}, \label{eq:MP_law}
\end{equation} 
where $\lambda^{\pm} = \left(1\pm\sqrt{c}\right)^2$. Furthermore, each $\alpha_i$ larger than $1+\sqrt{c}$ produces a spiked eigenvalue of $\Sb_w$, respectively, converging to
\begin{equation}
    \lambda_i=\alpha_i+c\alpha_i/(\alpha_i-1). \label{eq:asymp_spikes}
\end{equation}
\end{prop}
The proof will be omitted since it can be easily done by noticing that in this particular case,  $\Sb_{Rid}=M^{-1}\tilde\Cb_N^{1/2}\Zb^*\Zb\tilde\Cb_N^{1/2}$, where $\tilde\Cb_N=\frac{\Cb_N}{\xi_N}$ is the normalized matrix unitarily similar to
$$\frac{\diag(\alpha_1,\dots,\alpha_p,1,\dots,1)}{\xi'_N}$$
with 
$$\xi'_N:=\frac{\sum_{i=1}^p\alpha_i+N-p}{N}\xrightarrow[N\to\infty]{} 1.$$

In order to estimate $p$, we calculate the eigenvalues of $\Sb_w$ denoted as
$$\hat\lambda_1\ge \hat\lambda_2\ge \dots \ge \hat\lambda_N.$$
Note that for current model setting, the ratio $\hat\lambda_j/\hat\lambda_{j+1}$ is larger than $1+\varepsilon>1$ for $j<p$, and is close to $1$ for $j\ge p$. We thus choose three numbers $\gamma_1,\gamma_2,\gamma_3$ such that the event 
$$\left\{\frac{\hat\lambda_{p+i}}{\hat\lambda_{p+i+1}}<\gamma_i,\, i=1,2,3\right\}$$
occurs with probability no less than some threshold (e.g. $99.9\%$). We will determine $\gamma_i$ by Monte-Carlo simulation. Then
$$\hat p:=\inf\{k\ge 0 \tq \hat\lambda_{k+1}/\hat\lambda_{k+2}<\gamma_1, \; \hat\lambda_{k+2}/\hat\lambda_{k+3}<\gamma_2, \; \hat\lambda_{k+3}/\hat\lambda_{k+4}<\gamma_3 \}$$
is our proposed estimate for $p$.
\begin{Rq}
Theoretically, we can use
$$\hat p':=\inf\{k\ge 0 \tq \hat\lambda_{k+1}/\hat\lambda_{k+2}<\gamma_1 \}$$
as an estimate of $p$. However even though $\alpha_i$'s are distinct, the corresponding sample eigenvalues $\hat\lambda_i$ can still get so close that the algorithm may terminate prematurely and $\hat p'$ tends to under-estimate the real $p$. The triple tests in $\hat p$ reinforce the robustness of the estimator against such situation.
\end{Rq}

Once we have estimated $p$, we can further estimate the values of $\alpha_1,\dots, \alpha_p$. According to \eqref{eq:asymp_spikes}, a rough estimation is given by
\begin{equation}
    \hat\alpha_i:=\frac{(1-c+\hat\lambda_i)+\sqrt{(1-c+\hat\lambda_i)^2-4\hat\lambda_i}}{2}. \label{eq:rough_estim_alpha}
\end{equation}
However, note that $\Sb_{Rid}$ is different from the standard spike model by a factor of $\xi_N^{-1}$. The estimation error may be considerable if there are many large spikes. We can estimate $\xi_N$ in order to correct the error. Note that the exact asymptotic location of $\hat\lambda_{p+1}$ is at $\xi_N^{-1}(1+\sqrt{c})^2$, so we let
$$\hat\xi_N:=\frac{(1+\sqrt{c})^2}{\hat\lambda_{\hat p+1}}.$$
And then the corrected estimation of $\alpha_i$ is
\begin{equation}
    \hat\alpha_i:=\frac{(1-c+\hat\lambda_i\hat\xi_N)+\sqrt{(1-c+\hat\lambda_i\hat\xi_N)^2-4\hat\lambda_i\hat\xi_N}}{2}. \label{eq:correct_estim_alpha}
\end{equation}
We will see in the simulation that the accuracy of this estimator is satisfactory. 

We do some numerical simulations to test the efficiency and robustness of the two estimation procedures. We take $M=500, N=833$ ($c=0.6$) and $a=0.7$. Although the number $p$ and the spikes $\alpha_i$ are assumed to be fixed in the description of model, in order to add some challenge to the test, we pick $p$ randomly following Poisson distribution with parameter $4$, and then $\alpha_1,\dots,\alpha_p$ are independently and uniformly positioned in the interval $[3,10]$. Let $\sigma^2=1$. With this construction of $\Cb_N$, we know that if $p\ge 1$, almost surely the spikes $\alpha_1,\dots,\alpha_p$ are simple. 

Next we determine the three numbers $\gamma_1,\gamma_2,\gamma_3$ defined in the description of the algorithm. Independent samples are drawn $1000$ times under the spike-free model (or "white" model), that is, $\Cb_N=\Ib$, and the ratios $\lambda_{i}(\Sb_w)/\lambda_{i+1}(\Sb_w)$ for $i=1,2,3$ are recorded. Let $\gamma_i$ be the largest value of these ratios for $i=1,2,3$, respectively. By this method, we find 
$$\gamma_1=1.04418,\; \gamma_2=1.0353,\; \gamma_3=1.0294\,.$$

Using the above configurations, we make $1000$ independent realizations, and register the frequency (over the total $1000$ realizations) of each case in Table~\ref{table:recover_spikes}. For all the realizations such that $\hat p=p\ne 0$, we calculate the estimates $\hat\alpha_i, i=1,\dots,\hat p$ using \eqref{eq:correct_estim_alpha} and also their relative errors (RErr). We put the average RErr in Table~\ref{table:recover_spikes}. 
\begin{table}[ht]
\centering
\caption{Accuracy of $\hat p$ and $\hat\alpha_i$}
\begin{tabular}{cccccc}
\toprule
 & Proportion & & Proportion & &  Mean RErr of $\hat\alpha_i$ \\
 \hline
\multirow{2}{*}{$\hat p=p$} & \multirow{2}{*}{$99.8\%$} & $p\ne 0$ & $97.5\%$ & & $0.0031958$  \\
 & & $p=0$ & $2.3\%$ &  \\
\hline
\multirow{2}{*}{$\hat p>p$} & \multirow{2}{*}{$0.2\%$} & $p\ne 0$ & $0.2\%$ &  \\
    & & $p=0$ & $0$ &   \\
\hline
$\hat p<p$ & $0$ & &  &  \\
\bottomrule
\end{tabular}
\label{table:recover_spikes}
\end{table}

We can see that the estimator $\hat p$ has an accuracy of $99.8\%$, and the estimators $\hat\alpha_i$ are also accurate.

\subsection{PCA on time-correlated data matrix}\label{subsec:PCA}
The PCA is a widely used method for noise reduction and data compression. Given the data matrix $\Xb$, our aim in this section is to reduce the row dimension $N$ by removing the noise and preserving as much signal as possible contained in $\Yb$.

We recall the main steps of PCA on the row vectors of $\Yb$, if the original data $\Yb$ is available. 
\begin{enumerate}
    \item Calculate the eigenvalues and the associated eigenvectors of the sample covariance matrix $\Sb_{Y}=M^{-1}\Yb\Yb^*$.
    \item Estimate the number of principal components (PC) $p$ using the algorithm described in \S\ref{subsec:detect_number_signals}.
    \item Let $\vb_1,\dots,\vb_p$ be the eigenvectors of $\Sb_{Y}$ associated to the spikes. Then we get the PC of each vector $\yb_k$ by projecting it into the subspace generated by $(\vb_1,\dots,\vb_p)$, that is,
    $$\yb_{k,pc}:=\mathbf{P}\yb_k=\sum_{i=1}^p \langle \yb_k, \vb_i\rangle \vb_i.$$
    The PC of the matrix data is then
    $$\Yb_{pc}:=(\yb_{1,pc},\dots,\yb_{M,pc})=\mathbf{P}\Yb,$$
    where 
    $$\mathbf{P}:=\sum_{i=1}^p \vb_i\vb_i^*$$
    is the matrix of the orthogonal projection into the subspace generated by $\{\vb_1,\dots,\vb_p\}$. In this way we have removed the noise from the data.
    \item  We can further compress the data by expressing the vectors $\yb_{k,pc}$ as $p$-dimensional vectors in the coordinate system $\vb_1,\dots,\vb_p$, that is,
    $$\left.\yb_{k,pc}\right|_{(\vb_1,\dots,\vb_p)}=(\vb_1^* , \dots , \vb_p^*)^\tran\yb_k.$$
    So the compressed data matrix from $\Yb$ is
    $$\left.\Yb_{pc}\right|_{(\vb_1,\dots,\vb_p)}=(\vb_1^*, \dots , \vb_p^*)^\tran\Yb.$$
    Note that after compression, the dimension of the matrix becomes $p\times M$. In this way we have reduced the signal dimension from $N$ to $p$.
\end{enumerate}
  
When only $\Xb$ is observed, our main results suggest that we can proceed the PCA on the whitened matrix $\Yb_w$ to obtain the estimated PC of data. Note that in fact
$$\Sb_w=\frac{1}{M} \Yb_w\Yb_w^*,$$
so the PCA on $\Yb_w$ can be done using the algorithm described above by replacing $\Yb$ with $\Yb_w$, and $\Sb_{Y}$ with $\Sb_w$, respectively. Let $\yb_{w,k,pc}$ be the PC from the whitened data $\yb_{w,k}$, and let $\Yb_{w,pc}=(\yb_{w,1,pc},\dots,\yb_{w,M,pc})$. Then $\Yb_{w,pc}$ and $\Yb_{pc}$ are close in the following sense.
\begin{prop}\label{prop:consistency_pca_whitened}Let $\Yb$ be defined in \eqref{eq:def_ybj_pca} and $\Xb=\Yb\Rb_M^{1/2}$ with the conditions of Proposition~\ref{prop:whitening_consistence} hold. Assume also that the distance between $\alpha_i$ and $1$ are lower bounded. Then 
\begin{equation}
    \frac{1}{M}\|\Yb_{w,pc}-\Yb_{pc}/\sqrt{\xi_N}\|_F^2\xrightarrow[]{\as}0. \label{eq:consistency_pca_whitened}
\end{equation}
\end{prop}

Note that $\Yb$ contains $M$ columns (so there are $M$ signal vectors). Then \eqref{eq:consistency_pca_whitened} means that the PC's from whitened data and those from the original data are in average close to each other.

We take $N=1500, M=2500$ ($c=0.6$), $p=3$, and a matrix $\Ab$ of dimension $1500\times 3$. Although the matrix $\Ab$ in the model setting is deterministic, in our simulation program we have constructed it from three column Gaussian vectors of distributions $\gNr(0,\sigma_i^2\Ib)$ with $\sigma_1^2=0.1$, $\sigma_2^2=0.2$, $\sigma_3^2=0.3$ respectively, in order to better approximate to the reality. Finally, the singular values of $\Ab$ are $5.21, 20.89, 42.49$. Let $\Yb$ be a $N\times M$ random matrix whose rows are defined as in \eqref{eq:def_ybj_pca} with $\sigma^2=1$. Let $\Xb=\Yb\Rb_M^{1/2}$ where $\Rb_M$ is defined in \eqref{eq:def_r_ij_power_a} with $a=0.2$. Then we proceed the PCA on the whitened data $\Yb_w$ as described above, and get $\Yb_{w,pc}$. Since the data matrix $\Yb$ is also available in such simulation experiments, we can also calculate $\Yb_{pc}$, the PC of $\Yb$, and compare the row vectors of the two matrices. 

The comparison result is illustrated in Figure~\ref{fig:PCA}, where in \ref{fig:norm_diff}, we draw the Euclidean norm of $\yb_{w,k,pc}-\yb_{k,pc}/\sqrt{\xi_N}$, and in \ref{fig:cos_o_w}, we draw the real part of cosine similarity of $\yb_{w,k,pc}$ and $\yb_{k,pc}$, for $k=1,\dots,M$. The cosine similarity of two vectors $\ub,\vb\in\C^N$ is defined by
$$\cos(\ub,\vb):=\frac{\langle \ub,\vb\rangle}{\|\ub\|\|\vb\|},$$
and its real part represents the cosine similarity of $\ub,\vb$ regarded as real vectors in $\R^{2N}$. From the simulation result we can see that the the PC vectors from whitened data are close to those from the original data. The norms of differences are under $0.5$ for the majority part of the signal vectors, while the norms of the PC signal vectors themselves are around $10$. The real cosine similarities are close to $1$, which means that the directions of the two vectors under comparison are almost the same. 

We repeat the same experiments with the same parameters except for $a=0.95$. The results are shown in Figure~\ref{fig:PCA}, plots \ref{fig:norm_diff95} and \ref{fig:cos_o_w95}. We can see that although the LRD parameter $a$ is  quite close to $1$ (the process has a quite long memory), the PCA has a comparable accuracy to the previous case.
\begin{figure}[htbp]
\centering
\subfigure[The Euclidean norm of $\yb_{w,k,pc}-\yb_{k,pc}/\sqrt{\xi_N}$, with $a=0.2$. ]{\includegraphics[width=.48\textwidth]{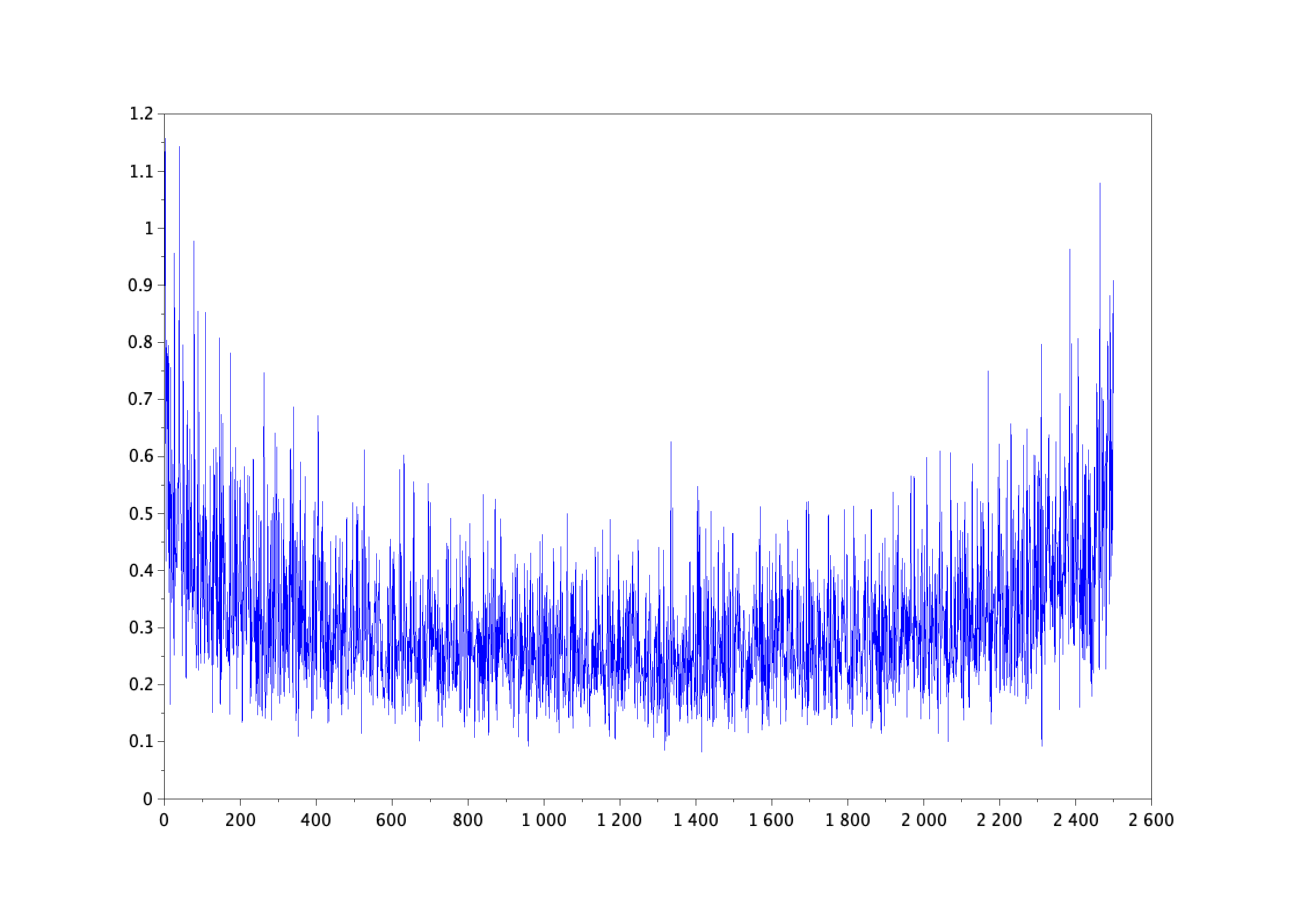}\label{fig:norm_diff}}
\subfigure[The real part of cosine similarity of $\yb_{w,k,pc}$ and $\yb_{k,pc}$, with $a=0.2$.]{\includegraphics[width=.48\textwidth]{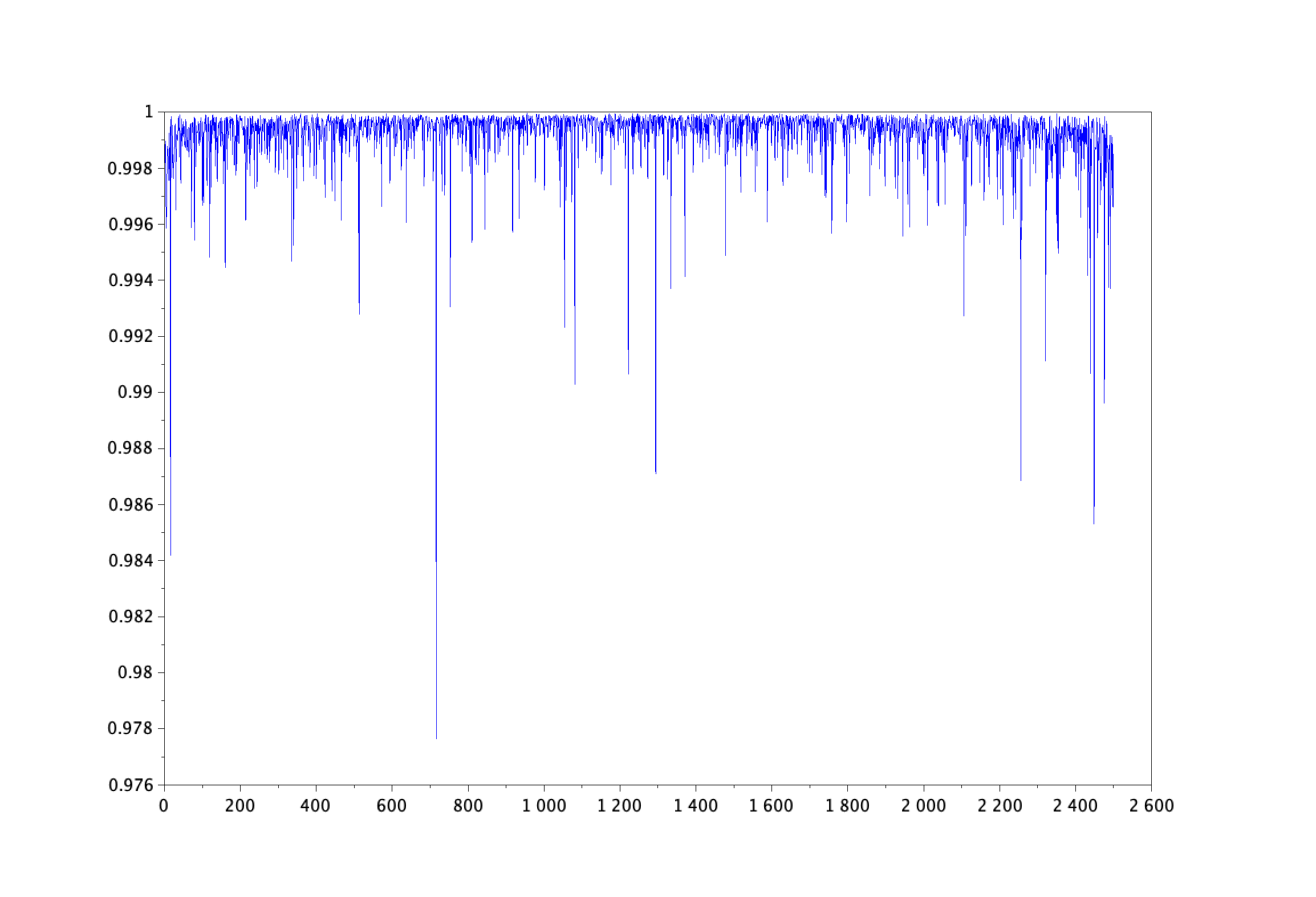}\label{fig:cos_o_w}}
\subfigure[The Euclidean norm of $\yb_{w,k,pc}-\yb_{k,pc}/\sqrt{\xi_N}$, with $a=0.95$ ]{\includegraphics[width=.48\textwidth]{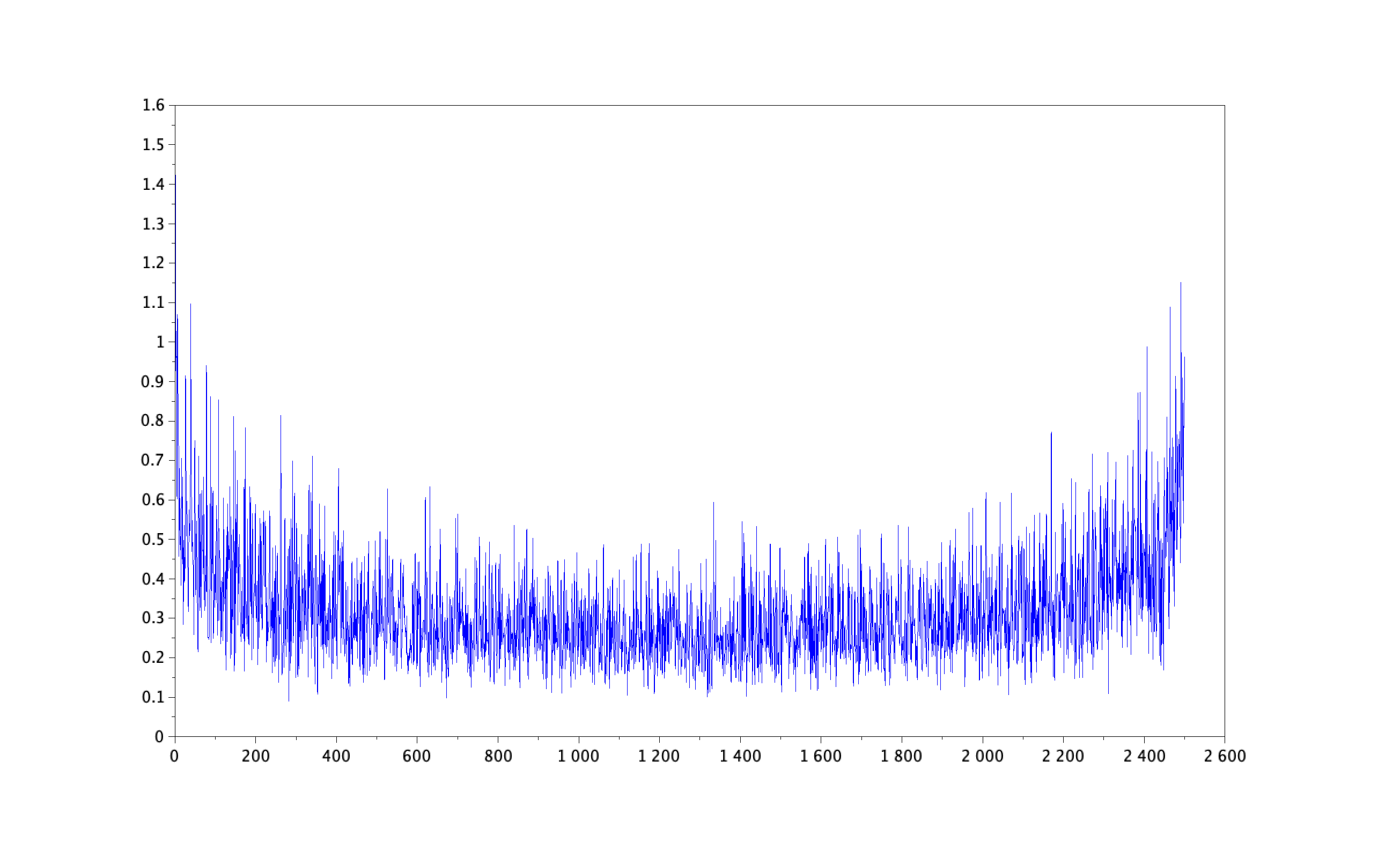}\label{fig:norm_diff95}}
\subfigure[The real part of cosine similarity of $\yb_{w,k,pc}$ and $\yb_{k,pc}$, with $a=0.95$]{\includegraphics[width=.48\textwidth]{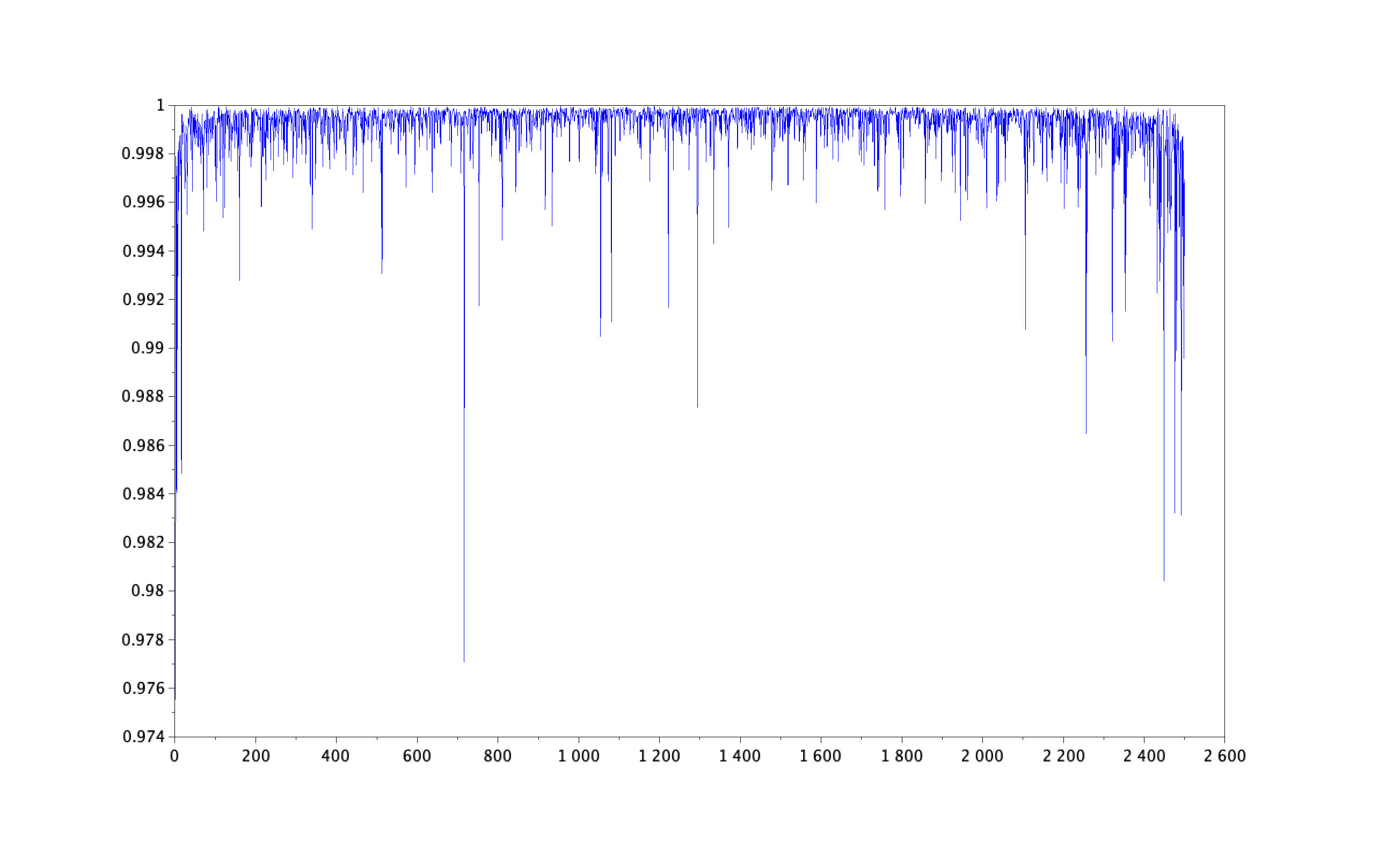}\label{fig:cos_o_w95}}
\caption{Comparison of PCA on the original and whitened data matrices.}\label{fig:PCA}
\end{figure}

\begin{Rq}Thanks to the separable structure $\Xb=\Yb\Rb_M^{1/2}$, if we are not interested in the original data but just want to reduce the underlying noise for compression, we only need to project the columns of $\Xb$ into the subspace generated by $\vb_{w,1},\dots,\vb_{w,p}$, the eigenvectors of $\Sb_w$, and get
$$\Xb_{w,pc}=\left(\sum_{i=1}^p \vb_{w,i}\vb_{w,i}^*\right)\Yb\Rb_M^{1/2},$$
which is the PC of $\Xb$.
\end{Rq}

\section{Numerical studies on inconsistency properties} \label{sec:num_illustrations}
In order to demonstrate the impact of new phenomena caused by LRD in the whitening procedure, and also to illustrate some of our conjectures, we present several numeric simulations in this section. 

Throughout this section, we assume that
\begin{equation}
    \Cb_N=\zeta_N^{-1}\diag(\alpha_1,\dots,\alpha_p, 1, \dots, 1) \label{eq:C_N_perturb_id}
\end{equation} 
where $p\ge 0$ is a fixed integer, $\alpha_i>1$, $i=1,\dots,p$ are some fixed positive numbers, and the normalization $\zeta_N=N^{-1}(\alpha_1+\dots+\alpha_p+N-p)$ is such that $\tr\Cb_N=N$. When $p=0$, $\Cb_N$ is identity. We assume also that the entries of $\Zb$ are i.i.d standard real or complex Gaussian.

\subsection{Norm inconsistency of $\hRb_M$ when $1/2<a<1$} \label{subsec:simu_norm_inconsistency}
We take $\Cb_N=\Ib$, $\Zb$ having i.i.d real standard Gaussian entries. In order to check the consistency of the unbiased estimate $\hRb_M$ with $\Rb_M$ in spectral norm, we take $a=0.9, 0.7, 0.5, 0.3, 0.1$ and $M=250, 500, 1000, 2000$ with $N=2M$. For each case we sample $500$ independent realizations, and list the medians of $\|\hRb_M-\Rb_M\|$ in Table~\ref{table:sp_norm_inconsistance}. 

\begin{table}[ht]
\centering
\caption{Medians of $\|\hRb_M-\Rb_M\|$.}
\begin{tabular}{crrrr}
\toprule
$a$ & $M=250$ & $M=500$ & $M=1000$ & $M=2000$ \\
\hline  
0.9 & 7.0658  &   10.9346 & 12.7499 & 16.5067 \\
0.7 & 3.2477  &    4.0702 &  4.6998 &  5.7822 \\
0.5 & 1.8272  &    1.9393 &  1.9315 &  1.9579 \\
0.3 & 1.1738  &    1.0211 &  1.0285 &  0.9318 \\
0.1 & 0.7494  &    0.6552 &  0.5792 &  0.4873 \\
\bottomrule
\end{tabular}
\label{table:sp_norm_inconsistance}
\end{table}
We know that if $\|\hRb_M-\Rb_M\|\to 0$ in probability, then the median must also converge to $0$. However from Table~\ref{table:sp_norm_inconsistance} we can see that when $a>0.5$, the median of $\|\hRb_M-\Rb_M\|$ is large and increasing with $M$ ($N=2M$). When $a=0.5$, which is the theoretical threshold of spectral norm consistency, the median of $\|\hRb_M-\Rb_M\|$ seems oscillating, neither increasing nor decreasing. When $a<0.5$, in which case we know that $\|\hRb_M-\Rb_M\|\to 0$ almost surely (see \eqref{eq:large_dev_normsp}), the medians are relatively small and tend to decrease with $M$.

\subsection{Ratio inconsistency of $\Rb_M^b$ and its impact on the whitening procedure} \label{subsec:ratio_inconsistency}
We have seen that a striking difference between the LRD and SRD situations is that the biased estimator $\Rb_M^b$ is not ratio consistent (Proposition~\ref{prop:inconsis_Rb}), but instead, it is ratio LSD consistent (Proposition~\ref{prop:lsd_consistent}). That means, only a small number of eigenvalues of the ratio $\Rb_M^{-1}\hRb_M^b$ deviate from $1$. But how many deviating eigenvalues are there, and how does this affect the applications? We now study these questions via  numerical experiments.

We take $M=1000, N=2000$ ($c=2$), $\Cb_N=\Ib$, $\Zb$ has real Gaussian entries, and $a=0.9$ (the same configuration as the first row and second column of Table~\ref{table:sp_norm_inconsistance}). We plot the histograms of the spectra of $(\hRb_M^b)^{-1}\Rb_M$ and $\hRb_M^{-1}\Rb_M$ in Figure~\ref{fig:Esd_ratio_bandu}. We note that the major part of eigenvalues of $(\hRb_M^b)^{-1}\Rb_M$ are close to $1$, but there are several extreme ones which are far away, 
the smallest at $0.3$, and the largest at $1.75$. In contrary, the spectrum of $\hRb_M^{-1}\Rb_M$ spreads in the interval $[0.91,1.12]$ much concentrated around $1$. 
\begin{figure}[htbp]
\centering
\subfigure[ESD of $(\hRb_M^b)^{-1}\Rb_M$.]{\includegraphics[width=0.47\linewidth]{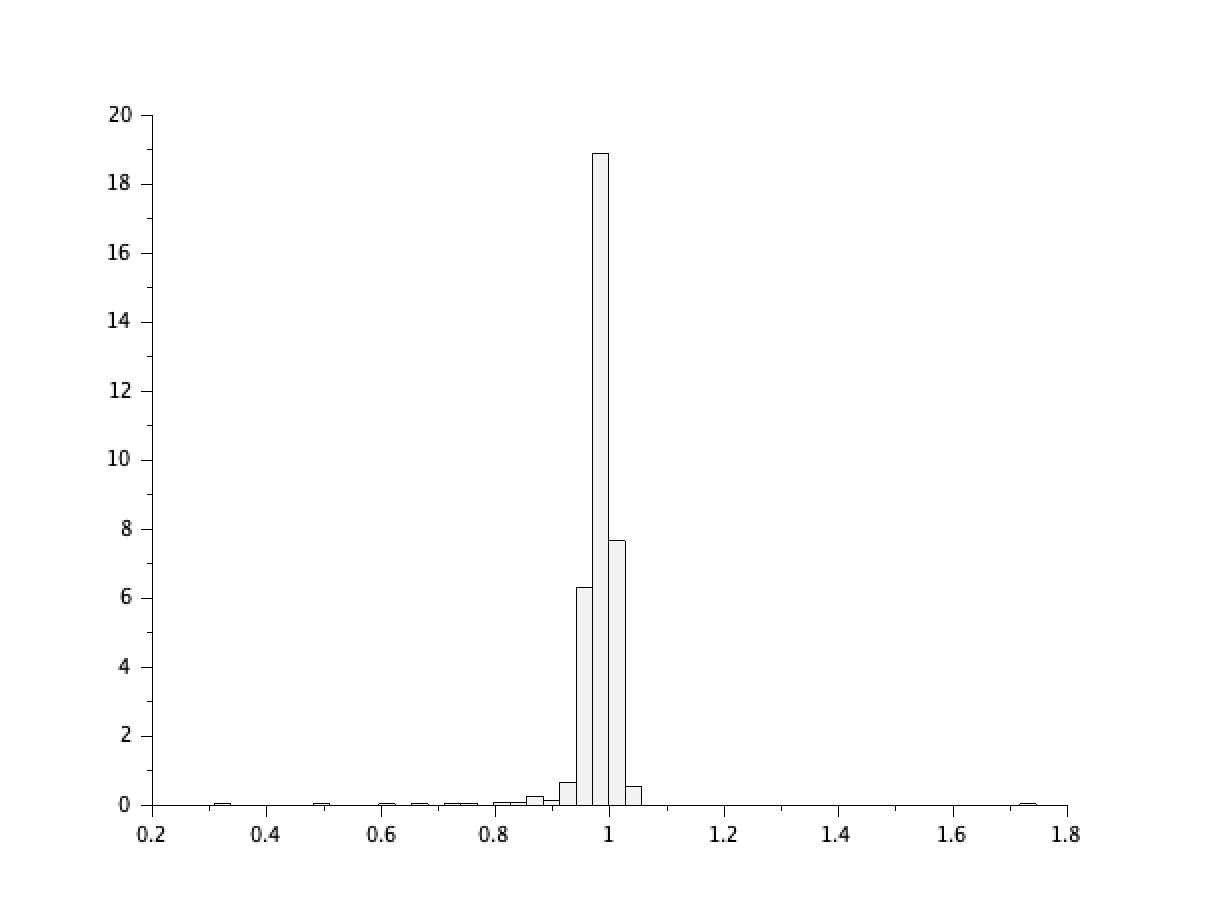}\label{fig:esd_Rbb}}
\subfigure[ESD of $\hRb_M^{-1}\Rb_M$.]{\includegraphics[width=0.47\linewidth]{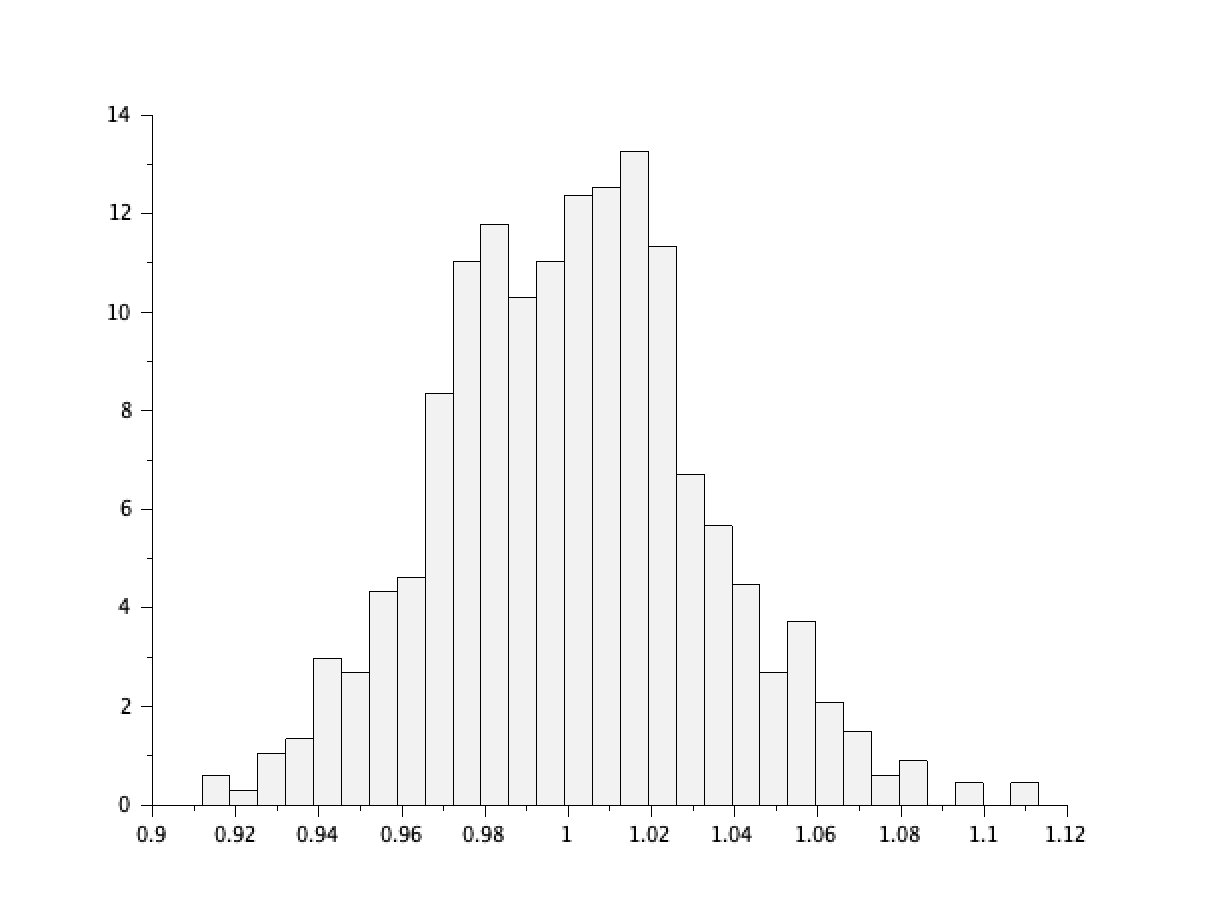}\label{fig:esd_Rbu}}
\caption{Comparison of ESD between $(\hRb_M^b)^{-1}\Rb_M$ and $\hRb_M^{-1}\Rb_M$.}
\label{fig:Esd_ratio_bandu}
\end{figure}

Because of the ratio inconsistency of $\hRb_M^b$, we may observe some extra "pseudo" spikes caused by the spikes of $\Rb_M^{1/2}(\hRb_M^b)^{-1}\Rb_M^{1/2}$, if we replace $\hRb_M$ with $\hRb_M^b$ in the whitening procedure. Let
$$\bSb_{wb}:=\frac{1}{M}\Xb^*(\hRb_M^b)^{-1}\Xb=\frac{1}{M}\Cb_N^{1/2}\Zb^*(\Rb_M^{1/2}(\hRb_M^b)^{-1}\Rb_M^{1/2})\Zb\Cb_N^{1/2},$$
which parallels the matrix $\Sb_w$ in \eqref{eq:Sw_pca} with this replacement. In order to better illustrate the pseudo spikes, when $N>M$, we will plot the ESD of its dual sample covariance matrix
$$\Sb_{wb}:=\frac{1}{N}(\hRb_M^b)^{-1/2}\Xb\Cb_N\Xb^*(\hRb_M^b)^{-1/2}.$$

We take $M=1000, N=8000$, $\Cb_N=\Ib$, and plot in Figure~\ref{fig:pseudo_spikes} the histogram of the ESD's of $\Sb_{wb}$ and also of the corresponding dual matrix $\Sb_w$ derived from $\Sb_w$. We can see that when $\Cb_N$ has no spikes, some unexpected spikes are observed in the ESD of $\Sb_{wb}$, whereas the corresponding $\Sb_w$ does not have this problem. Note also that this phenomenon occurs only with very large ratio $c=N/M$. When we take $N=3000$ or $N=800$ instead, the pseudo spikes disappear, see Figure~\ref{fig:esd_bSbw_3000} and \ref{fig:esd_bSbw_800}. 
\begin{figure}[htbp]
\centering
\subfigure[ESD of $\Sb_{wb}$ with $c=N/M=8$.]{\includegraphics[width=0.47\linewidth]{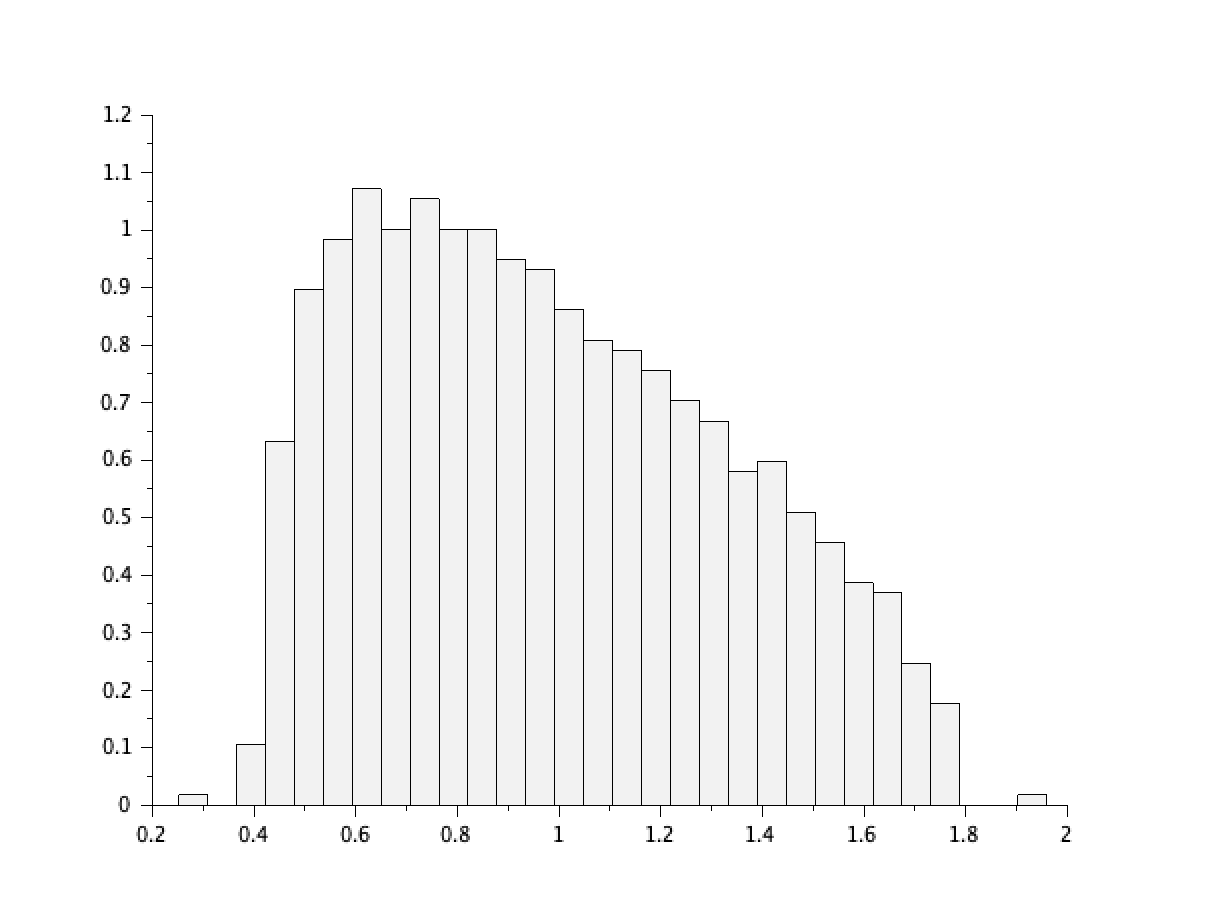}\label{fig:esd_bSbwb}}
\subfigure[ESD of $\Sb_w$ with $c=N/M=8$.]{\includegraphics[width=0.47\linewidth]{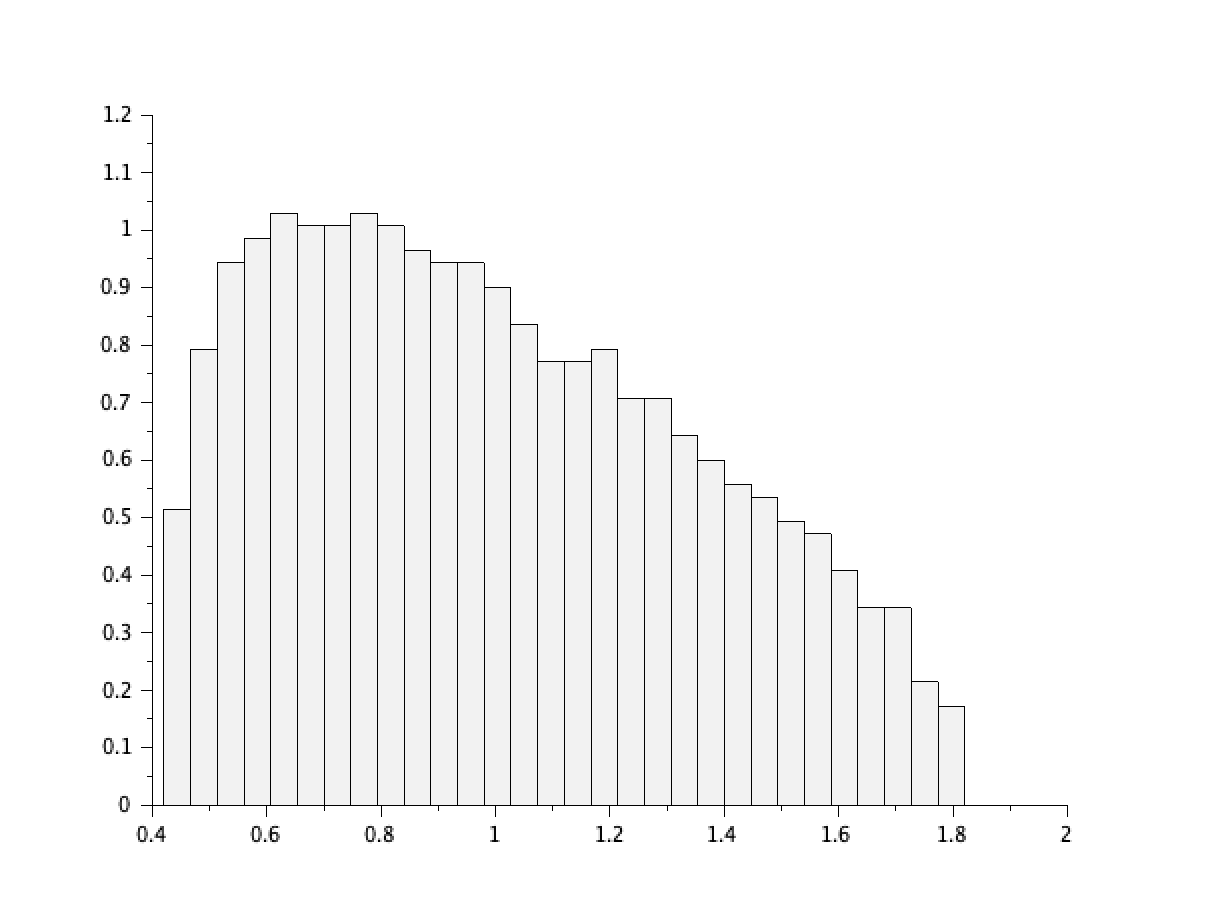}\label{fig:esd_bSbw}}
\subfigure[ESD of $\Sb_{wb}$ with $c=N/M=3$.]{\includegraphics[width=0.47\linewidth]{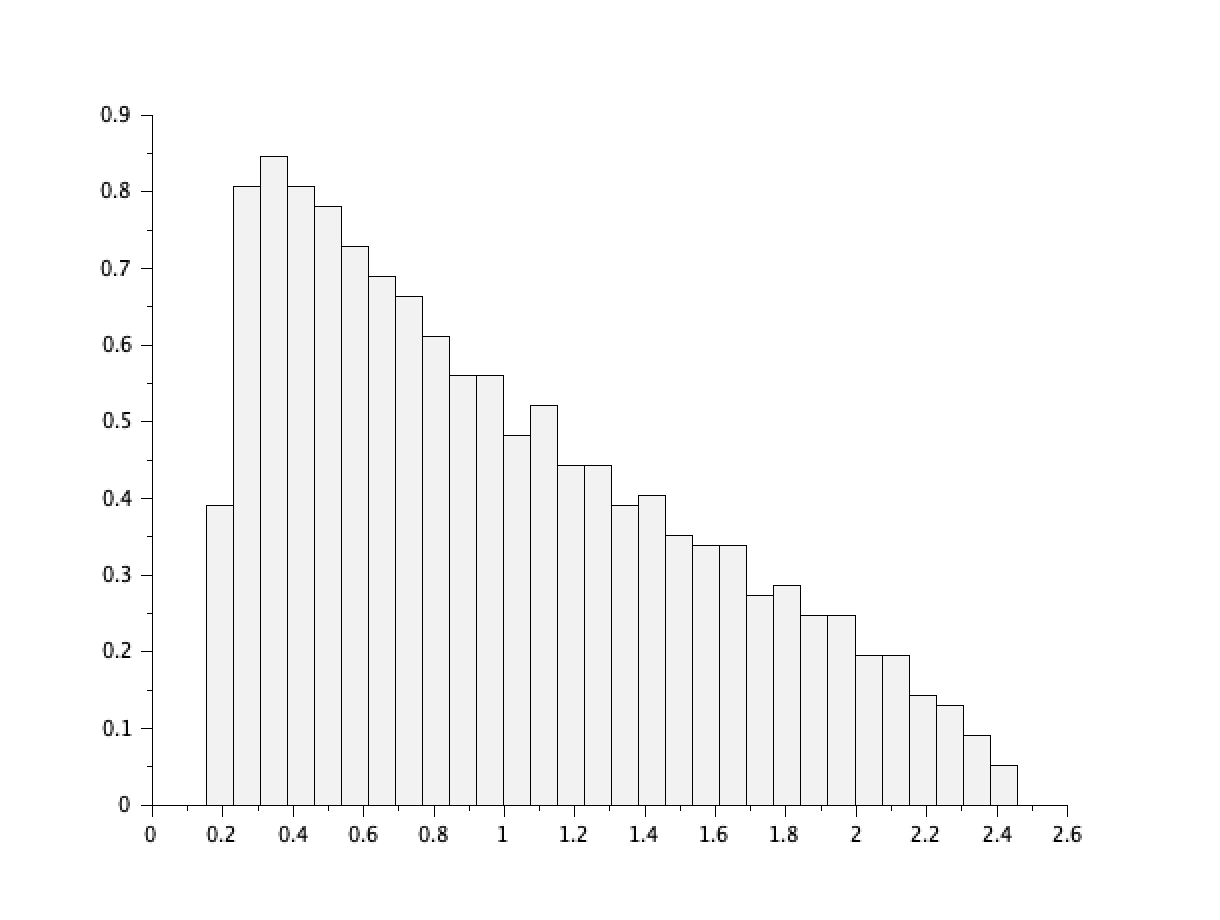}\label{fig:esd_bSbw_3000}}
\subfigure[ESD of $\bSb_{wb}$ with $c=N/M=0.8$.]{\includegraphics[width=0.47\linewidth]{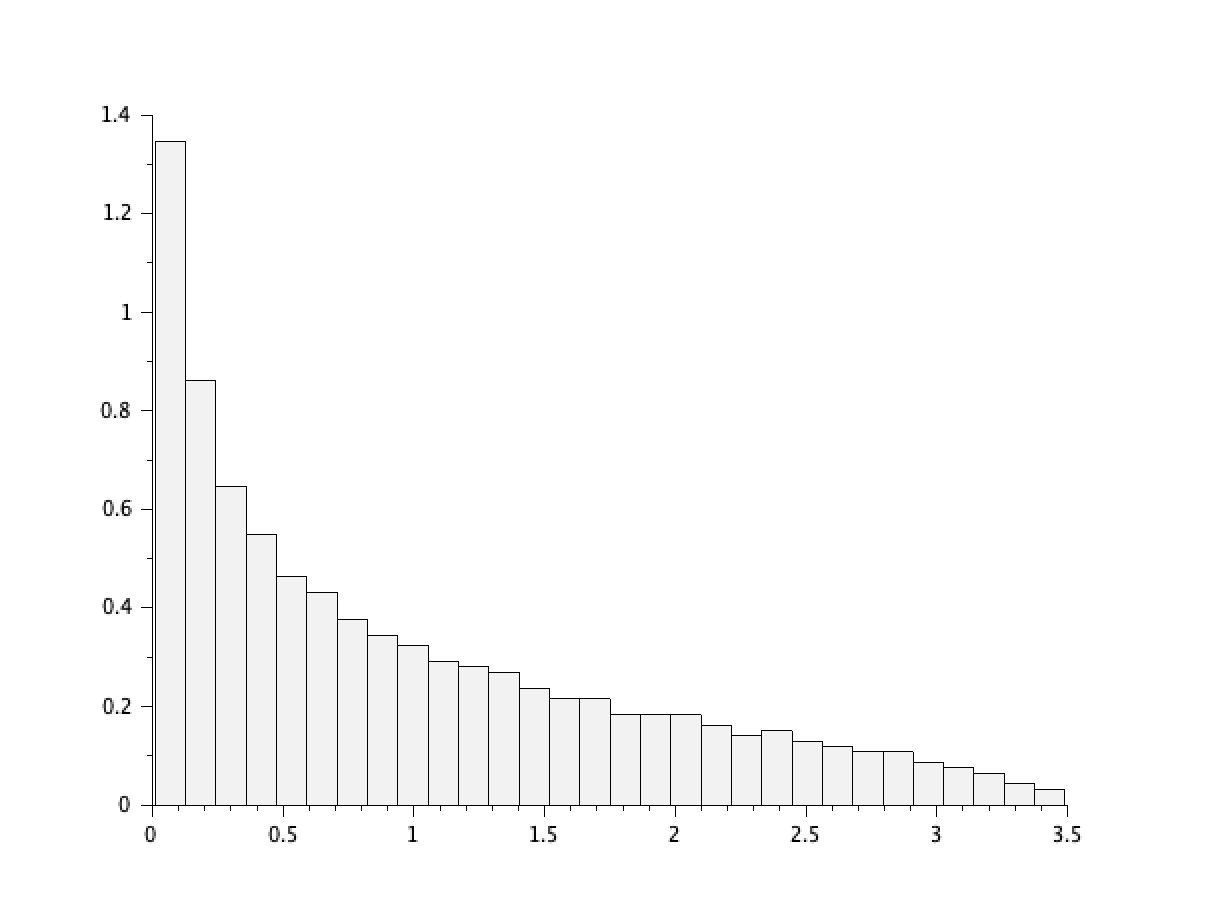}\label{fig:esd_bSbw_800}}
\caption{Inconsistency of $\Rb_M^b$ may cause pseudo spikes under some conditions.}
\label{fig:pseudo_spikes}
\end{figure}

Here is a heuristic explanation. From the ratio LSD consistency described in Proposition~\ref{prop:lsd_consistent}, the ratio inconsistency described in Proposition~\ref{prop:inconsis_Rb} and the numeric simulations in this section, we can think of $(\Rb_M^b)^{-1/2}\Rb_M(\Rb_M^b)^{-1/2}$ asymptotically as a finite perturbation of identity. Thus the appearance or disappearance of pseudo spikes can be explained by the spikes of separable model
$$\Sb_{sep}:=\frac{1}{M}\Cb_N^{1/2}\Zb\mathbf{\Sigma}\Zb^*\Cb_N^{1/2}$$
where $\mathbf{\Sigma}$ can be a positive deterministic Hermitian matrix. Largest eigenvalues of this matrix obeys a phase transition phenomenon as established in \cite{ding2021spiked}. 

\section{Proof of Theorem~\ref{th:main_th}}  \label{sec:proof_main_th}
\subsection{Some preliminaries} \label{subsec:preliminary_pr_main}
As the matrices $\Rb_M^{-1/2}\hRb_M\Rb_M^{-1/2}$ and $\hRb_M\Rb_M^{-1}$ have the same eigenvalues, we have
$$\|\Rb_M^{-1/2}\hRb_M\Rb_M^{-1/2}-\xi_N\Ib\|=\max_i\{|\lambda_i(\hRb_M\Rb_M^{-1}-\xi_N\Ib)|\}.$$
Then the idea of proof is to estimate the range of eigenvalues of the matrix $\hRb_M\Rb_M^{-1}$. 

The following lemma connects the spectrum of $\hRb_M\Rb_M^{-1}$ with the ratio of spectral densities of the two Toeplitz matrices $\hRb_M$ and $\Rb_M$. It was first proved in \cite{grenander2001toeplitz} and extended to integrable spectral densities in \cite[Theorem~2.1]{serra1998extreme}.

\begin{lemma}\label{lem:hRRinv}Let $\Rb_{1,M}, \Rb_{2,M}$ be two $M\times M$ Toeplitz matrices generated by positive spectral densities $f_1, f_2\in L^1(0,2\pi)$, respectively. Then for any $M\ge 1$,
$$\spec(\Rb_{1,M}\Rb_{2,M}^{-1})\subset \left[\essinf_{\theta\in[0,2\pi]}\frac{f_1(\theta)}{f_2(\theta)},\quad \esssup_{\theta\in[0,2\pi]}\frac{f_1(\theta)}{f_2(\theta)}\right].$$
\end{lemma}

By this lemma, the spectral densities of the two Toeplitz matrices $\hRb_M$ and $\Rb_M$ are important. We note that $\hRb_M$ is random and depends on $N,M$, then so must be its spectral density, and the coefficients of orders higher than $M-1$ can be arbitrary. For each $N$ and $M$, we define
\begin{equation}
    \hf_M(\theta):= \xi_N f(\theta)+\sum_{n=-M+1}^{M-1}(\hat r_n-\xi_N r_n)e^{\im n \theta} \label{eq:hatf}
\end{equation}
where $f$ is the spectral density of $\Rb_M$. Note that the Fourier coefficients of $\hf_M$ are $\hat r_m$ for $-M+1\le m\le M-1$, thus for this particular $N$ and $M$, $\hf_M$ is the spectral density of $\hRb_M$, and by Lemma~\ref{lem:hRRinv}, the eigenvalues of $\hRb_M\Rb_M^{-1}$ are in the interval
$$\left[\essinf_{\theta\in[0,2\pi]}\frac{\hf_M(\theta)}{f(\theta)},\quad \esssup_{\theta\in[0,2\pi]}\frac{\hf_M(\theta)}{f(\theta)}\right].$$
Thus for any $x>0$, we have
\begin{equation}
     \P\left(\|\Rb_M^{-1/2}\hRb_M\Rb_M^{-1/2}-\xi_N\Ib\|>x\right) \le  \P\left(\esssup_{\theta\in[0,2\pi]}\left|\frac{\hf_M(\theta)}{f_M(\theta)}-\xi_N\right|>x\right). \label{eq:to_ctrl_prob}
\end{equation}
Let
\begin{equation}
    \Upsilon_M(\theta):=\sum_{n=-M+1}^{M-1}r_n e^{\im n\theta}, \quad \hUpsilon_M(\theta):=\sum_{n=-M+1}^{M-1}\hat r_n e^{\im n\theta}. \label{eq:def_Upsilon}
\end{equation}
Then $\hf_M(\theta)-\xi_N f_M(\theta)=\hUpsilon_M(\theta)-\xi_N\Upsilon_M(\theta)$. Recall that $\E\hUpsilon_M(\theta)=\xi_N\Upsilon_M(\theta)$ for any $\theta\in [0,2\pi]$.
Then the RHS of \eqref{eq:to_ctrl_prob} becomes
\begin{equation}
    \P\left(\esssup_{\theta\in[0,2\pi]}\frac{|\hUpsilon_M(\theta)-\E\hUpsilon_M(\theta)|}{f(\theta)}>x\right). \label{eq:sup_relative_error}
\end{equation}
This can be considered as the probability of large relative error of the estimation $\hUpsilon_M(\theta)$ with respect to $f(\theta)$. We will use a similar discretization strategy as \cite{vinogradova2015estimation}. Let
$$0<\theta_1<\theta_2<\cdots < \theta_m < 2\pi$$
be an appropriate mesh of $(0,2\pi)$, which will be precised later, then a key step is to estimate the probability 
\begin{equation}
    \P\left(|\hUpsilon_M(\theta_j)-\E\hUpsilon_M(\theta_j)|>xf(\theta_j)\right) \label{eq:large_deviation_1}
\end{equation}
for each $\theta_j$. This will be done in \S\ref{subsec:large_dev_ind_theta}. 

\subsection{Relative error bound of $\hUpsilon_M(\theta)$ for individual $\theta$}
\label{subsec:large_dev_ind_theta}
The goal of this subsection is to prove the following Proposition~\ref{prop:indiv_theta_dev}.

\begin{prop}\label{prop:indiv_theta_dev} Let $K_1>0$ be an arbitrary positive constant. In both Gaussian case and spherical case, there exists $K_2>0$ depending on $K_1$, such that for any $x\in (0,CK_1/\kappa)$, for large enough $M$, any $N\ge 1$, and any $\theta\in(0,2\pi)$, we have
\begin{equation}
    \P\left(\left|\hUpsilon_M(\theta)-\E\hUpsilon_M(\theta)\right|>xf(\theta)\right)\le 2\exp\left(-\frac{Nx^2}{CK_2\log^2 M}\right). \label{eq:deviation_bound_indiv}
\end{equation}
\end{prop}

We will prove Proposition~\ref{prop:indiv_theta_dev} separately for complex Gaussian, complex spherical, and real cases. Before that, we still need some preliminary works.

Denote
\begin{equation}
    \Db_M(\theta):=\diag(1,e^{-\im\theta},\dots,e^{-\im(M-1)\theta}), \quad \Bb_M:=\left(\frac{1}{M-|i-j|}\right)_{i,j=0}^{M-1}, \label{eq:def_Db_Bb}
\end{equation}
and
\begin{equation}
    \Qb_M(\theta):=\Rb_M^{1/2}\Db_M(\theta)\Bb_M\Db_M^*(\theta)\Rb_M^{1/2}. \label{eq:def_QT}
\end{equation}
Then from Lemma 7 and 8 and (9) of \cite{vinogradova2015estimation}, under \ref{ass:C_diag}, we have
\begin{equation}
    \hUpsilon_M(\theta)=\frac{1}{N}\tr\Cb_N^{1/2}\Zb\Qb_M(\theta)\Zb^*\Cb_N^{1/2}=\frac{1}{N}\sum_{n=1}^N c_n\zb_n\Qb_M(\theta)\zb_n^*, \label{eq:hatU_tr}
\end{equation}
From \ref{ass:Gaussian}, $\zb_n^\tran$ can be a real or complex vector. We only give the complete proof for the complex case, and list the differences between real and complex cases to ease the adaption for the real case. Thus, let us first assume that $\zb_n^\tran$ are complex Gaussian or uniformly distributed on the complex sphere. Let $\sigma_1\ge \sigma_2 \ge \dots \ge \sigma_M$ be the eigenvalues of $\Qb_M(\theta)$ (Warning: $\Qb_M(\theta)$ may be indefinite). By the unitary invariance of $\zb_n^\tran$, we have
\begin{equation}
    \hUpsilon_M(\theta)-\E\hUpsilon_M(\theta)\eqinlaw \frac{1}{N}\sum_{n=1}^Nc_n\sum_{m=1}^M \sigma_m(|Z_{n,m}|^2-1). \label{eq:eqinlaw}
\end{equation}

In light of \eqref{eq:eqinlaw}, in order to prove \eqref{eq:deviation_bound_indiv}, it is equivalent to prove 
\begin{equation}
    \Pc:=\P\left(\left|\frac{1}{N}\sum_{n=1}^Nc_n\sum_{m=1}^M \sigma_m(|Z_{n,m}|^2-1)\right|>xf(\theta)\right)\le 2\exp\left(-\frac{Nx^2}{CK_2\log^2 M}\right). \label{eq:ineq_equiv}
\end{equation}
For this, we find it crucial to estimate $\tr\Qb_M^2(\theta)=\sum_{m=1}^M\sigma_m^2$. We now state the following proposition, whose proof is provided in \S\ref{subsec:proof_estimation_trQb2}.

\begin{prop} \label{prop:traceQ_bound_0} Let $\Qb_M(\theta)$ be defined as \eqref{eq:def_QT} with Toeplitz matrix $\Rb_M$ whose spectral density $f$ satisfies \ref{ass:spectral_density}, \ref{ass:spectral_bound_below}, \ref{ass:spec_den_behav_0}. Then
\begin{equation}
    \frac{\tr\Qb_M^2(\theta)}{f^2(\theta)\log^2 M} \label{eq:trQ2_log2}
\end{equation}
is uniformly bounded in $\theta\in [-\pi,\pi]\backslash\{0\}$ and $M > 1$.
\end{prop}

\subsubsection{Proof of Proposition~\ref{prop:indiv_theta_dev}, complex Gaussian case} 
Let 
$$\sigma_i'=\frac{\sigma_i}{\sqrt{\sum_{m=1}^M\sigma_m^2}}, \quad i=1,\dots,M.$$
Then
$$\mathcal P=\P\left(\left|\sum_{n=1}^Nc_n\sum_{m=1}^M \sigma'_m(|Z_{n,m}|^2-1)\right|>\frac{Nxf(\theta)}{\sqrt{\sum_{m=1}^M \sigma_m^2}}\right).$$
By Proposition~\ref{prop:traceQ_bound_0}, there exists a constant $K>0$ such that
$$\frac{Nxf(\theta)}{\sqrt{\sum_m\sigma_m^2}}\ge \frac{Nx}{\sqrt{K}\log M}$$
for any $\theta$ and $M > 1$. Then 
\begin{equation}
    \Pc \le \P\left(\left|\sum_n c_n\sum_m\sigma_m'(|Z_{n,m}|^2-1)\right|>\frac{Nx}{\sqrt{K}\log M}\right). \label{eq:proba_simplify_Gaussian}
\end{equation} 
Then we only need to estimate the RHS of \eqref{eq:proba_simplify_Gaussian} with $\sum_m (\sigma'_m)^2=1$. Let
$$\Pc_i:=\P\left((-1)^{i-1}\sum_n c_n\sum_m\sigma_m'(|Z_{n,m}|^2-1)>\frac{Nx}{\sqrt{K}\log M}\right),\quad i=1,2.$$
Then $\Pc\le \Pc_1+\Pc_2$, and the estimation of $\Pc_1$ and $\Pc_2$ is similar, we only need to estimate $\Pc_1$. 

Using Chernoff bound, for any $\tau>0$, we have
\begin{equation}
    \Pc_1\le \exp\left(-\frac{Nx\tau}{\sqrt{K}\log M}+\log\E e^{\sum_n \tau c_n\sum_m\sigma_m'(|Z_{n,m}|^2-1)}\right).
\end{equation}
Note that the rows $\zb_n$ of $\Zb$ are i.i.d across $n$, we then have
\begin{equation}
    \Pc_1\le \exp\left(-\frac{Nx\tau}{\sqrt{K}\log M}+\sum_n\Phi_M(\tau c_n)\right), \label{eq:Pc1_le_Phi_M}
\end{equation}
where $\Phi_M$ is the cumulant generating function of $\sum_m\sigma_m'(|Z_{n,m}|^2-1)$:
$$\Phi_M(z):=\log\E e^{z\sum_m\sigma_m'(|Z_{n,m}|^2-1)}.$$
\begin{lemma}\label{lem:Phi_le_Az2_Gaussian}When $Z_{n,m}$ are i.i.d. standard complex Gaussian variables, there exists $A>0$ and $\varepsilon>0$ such that when $|z|<\varepsilon$, we have
$$|\Phi_M(z)|\le A|z|^2$$
for any $M>1$.
\end{lemma}
\begin{proof}
Let $\phi$ be the cumulant generating function of $|Z_{n,m}|^2-1$:
$$\phi(z):=\log \E e^{z (|Z_{n,m}|^2-1)}.$$
Then as $Z_{n,m}$ are i.i.d. standard complex Gaussian, we have
$$\Phi_M(z)=\sum_m\phi(z\sigma_m'),$$
and
$$\phi(z)=-z-\log(1-z).$$
By the Taylor's expansion $\log(1-z)=z-z^2/2+z^3/3-\cdots$, choosing an arbitrary $\varepsilon\in (0,1)$, there exists $A>0$ such that for any $|z|\le \varepsilon$, we have
$$|\phi(z)|=|z|^2|1/2-z/3+\cdots|\le A |z|^2.$$
Thus
$$|\Phi_M(z)|\le \sum_m|\phi(z\sigma_m')|\le \sum_m A|z|^2(\sigma'_m)^2=A|z|^2.$$
\end{proof}

From \eqref{eq:Pc1_le_Phi_M} and Lemma~\ref{lem:Phi_le_Az2_Gaussian}, for any $\tau>0$ such that $|\tau c_n|\le \varepsilon$, we have
\begin{equation}
    \Pc_1 \le \exp\left(-\frac{\tau Nx}{\sqrt{K}\log M}+A\tau^2 \sum_n c_n^2\right). \label{eq:Pc_1_Atau2cn2_Gau}
\end{equation}
Noting that $\sum_n c_n^2\le CN$ by \ref{ass:C_moments}, we then have
\begin{equation}
    \Pc_1 \le \exp\left(-\frac{\tau Nx}{\sqrt{K}\log M}+CA N\tau^2\right). \label{eq:Pc1_le_CNtau2}
\end{equation}
If we can take
$$\tau=\frac{x}{2CA\sqrt{K}\log M},$$
we will minimize the RHS of \eqref{eq:Pc1_le_CNtau2} and get
\begin{equation}
    \Pc_1 \le \exp\left(-\frac{Nx^2}{4KCA\log^2 M}\right). \label{eq:Pc1_4KCAlog2M}
\end{equation}
In order to validate \eqref{eq:Pc1_4KCAlog2M}, we have to keep $|\tau c_n|\le \varepsilon$ for all $n$. Note that $|c_n|\le \kappa \log M$ by \ref{ass:C_bound}, we only need to keep $\tau \le \varepsilon/(\kappa \log M)$. That is, we only need to keep
$$x\le 2\varepsilon CA\sqrt{K}/\kappa.$$

Taking $K_1=2\varepsilon A\sqrt{K}$ and $K_2=4KA$, we conclude that, for any $x\in (0,CK_1/\kappa)$, 
\begin{equation}
    \Pc_1\le \exp\left(-\frac{Nx^2}{CK_2\log^2 M}\right). \label{eq:deviation_bound_indiv1}
\end{equation}
Note that $K$ or $A$ can be adjusted to a larger constant, which means that $K_1$ can be arbitrarily large, and $K_2$ should be adjusted correspondingly. This is exactly the statement of Proposition~\ref{prop:indiv_theta_dev}. We have thus proved the proposition for complex Gaussian case.

\subsubsection{Proof of Proposition~\ref{prop:indiv_theta_dev}, complex spherical case} When $\zb_n$ follows the uniform distribution on the sphere $\{\zb\in\C^M\tq \|\zb\|=\sqrt{M}\}$. Then $\|\zb_n\|^2=M$. We have
$$\begin{aligned}\hUpsilon_M(\theta)-\E\hUpsilon_M &= \zb_n\Qb_M(\theta)\zb_n^*-\frac{\|\zb_n\|^2}{M}\tr\Qb_M(\theta) \\
    &= \zb_n\left(\Qb_M(\theta)-\frac{\tr\Qb_M(\theta)}{M}\Ib\right)\zb_n^* \\
    &\eqinlaw \sum_{m=1}^M \left(\sigma_m-\frac{\tr\Qb_M(\theta)}{M}\right)|Z_{n,m}|^2.
\end{aligned}$$
Write 
$$\sigma_m'=\frac{\sigma_m-\tr\Qb_M(\theta)/M}{\sqrt{\sum_m(\sigma_m-\tr\Qb_M(\theta)/M)^2}},$$
then $\sum_m \sigma_m'=0$, $\sum_m {\sigma'}_m^2=1$, and $\Pc$ defined in \eqref{eq:ineq_equiv} becomes
$$\Pc=\P\left(\left|\sum_{n=1}^Nc_n\sum_{m=1}^M \sigma_m'|Z_{n,m}|^2\right|>\frac{Nxf(\theta)}{\sqrt{\sum_m(\sigma_m-\tr\Qb_M(\theta)/M)^2}}\right).$$
Using Proposition~\ref{prop:traceQ_bound_0} again, there exists some constant $K>0$ such that
$$\sum_{m=1}^M\left(\sigma_m-\frac{\tr\Qb_M(\theta)}{M}\right)^2\le \tr\Qb_M^2(\theta)\le Kf^2(\theta)\log^2 M.$$
Then
$$\Pc \le \P\left(\left|\sum_{n=1}^Nc_n\sum_{m=1}^M \sigma_m'|Z_{n,m}|^2\right|>\frac{Nx}{\sqrt{K}\log M}\right).$$
Similar to the proof in the Gaussian case, we define 
$$\Pc_i=\P\left((-1)^{i-1}\sum_{n=1}^Nc_n\sum_{m=1}^M \sigma_m'|Z_{n,m}|^2>\frac{Nx}{\sqrt{K}\log M}\right), \quad i=1,2$$
and we just need to estimate $\Pc_1$. Using Chernoff bound,
for any $\tau>0$, we have
\begin{equation}
    \Pc_1\le \exp\left(-\frac{Nx\tau}{\sqrt{K}\log M}+\sum_n\Phi_M(\tau c_n)\right), \label{eq:Pc1_le_Phi_M_sph}
\end{equation}
where 
$$\Phi_M(z):=\log\E e^{z\sum_m\sigma'_m|Z_{n,m}|^2}=\frac{M}{2(M+1)}z^2+\cdots.$$
\begin{lemma}\label{lem:Phi_le_Az2_Spherical}When $\zb_n^\tran$ are i.i.d uniformly distributed on the complex sphere $\{\zb\in\C^M \tq \|\zb\|=\sqrt{M}\}$, there exists $A>0$ and $\varepsilon>0$ such that when $|z|<\varepsilon$, we have
$$|\Phi_M(z)|\le A|z|^2$$
for any $M>1$.
\end{lemma}
\begin{proof}
On the one hand, by the Taylor's expansion of $e^{\Phi_M(z)}$, and the fact that $\E\sum_m\sigma'_m|Z_{n,m}|^2=\sum_m \sigma'_m=0$, we have
\begin{equation}
    \E \exp\left(z\sum_m\sigma'_m|Z_{n,m}|^2\right)=1+\sum_{k=2}^\infty \frac{z^k}{k!}\E\left(\sum_m\sigma'_m|Z_{n,m}|^2\right)^k, \label{eq:taylor_exp_spheric}
\end{equation}
On the other hand, let $\gb=(g_1,\dots,g_M)^\tran \in\C^M$ be a standard complex Gaussian vector. Then since $\gb$ is spherically symmetric, we have $\gb\eqinlaw \|\gb\|\zb_n/\sqrt{M}$ where $\|\gb\|$ and $\zb_n$ are independent (see e.g.\cite{fang2018symmetric}). Then
\begin{equation}
    \begin{aligned} \E\exp\left(z\sum_m\sigma'_m|g_m|^2\right) &= \E\exp\left(z\frac{\|\gb\|^2}{M}\sum_m\sigma'_m|Z_{n,m}|^2\right) \\
    &= \sum_{k=0}^\infty \frac{z^k\E\|\gb\|^{2k}}{k!M^k}\E\left(\sum_m\sigma'_m|Z_{n,m}|^2\right)^k.
 \end{aligned} \label{eq:Gaussian_sphere_mntgen}
\end{equation}
Since we know that for Gaussian variables $g_m$, 
$$\E\exp\left(z\sum_m\sigma'_m|g_m|^2\right)=\prod_{m=1}^M\frac{1}{1-z\sigma'_m}= \exp\left(\sum_{m=1}^M\log(1-z\sigma'_m)\right),$$
comparing to \eqref{eq:Gaussian_sphere_mntgen} we get
\begin{equation}
    \sum_{k=0}^\infty \frac{z^k\E\|\gb\|^{2k}}{k!M^k}\E\left(\sum_m\sigma'_m|Z_{n,m}|^2\right)^k=\exp\left(\sum_{m=1}^M\log(1-z\sigma'_m)\right).\label{eq:gaussian_series}
\end{equation}
Note that $\sum_m \sigma'_m=0$, $\sum_m(\sigma'_m)^2=1$ and $|\sigma'_m|\le 1$. From the proof of Lemma~\ref{lem:Phi_le_Az2_Gaussian}, for an arbitrary $\epsilon\in (0,1)$, there exists $A_\epsilon$ such that $|\log(1-z\sigma'_m)+z\sigma'_m|\le A_\epsilon |z|^2(\sigma'_m)^2$ for any $|z|\le \epsilon$. Thus for these $z$ we have 
\begin{equation}
    \left|\exp\left(\sum_{m=1}^M\log(1-z\sigma'_m)\right)\right|\le
    \exp(A_\epsilon|z|^2).\label{eq:bound_exp_sum}
\end{equation}
Applying Cauchy's integration formula to \eqref{eq:gaussian_series} and using \eqref{eq:bound_exp_sum}, for each $k\ge 0$, we have
$$\left|\frac{\E\|\gb\|^{2k}}{k!M^k}\E(\sum_m\sigma'_m|Z_{n,m}|^2)^k\right|= \frac{1}{2\pi}\left|\int_{|z|=\epsilon}\frac{1}{z^{k+1}}\prod_{m=1}^M\frac{1}{1-z\sigma'_m}\dd z\right|\le \frac{e^{A_\epsilon \epsilon^2}}{\epsilon^k}.$$
Note that
$$\frac{\E\|\gb\|^{2k}}{M^k}=\frac{M(M+1)\cdots(M+k-1)}{M^k}\ge 1,$$
we get
$$\left|\frac{1}{k!}\E(\sum_m\sigma'_m|Z_{n,m}|^2)^k\right|\le \frac{e^{A_\epsilon \epsilon^2}}{\epsilon^k}.$$
Then for any $|z|\le \epsilon/2$, we have
$$\left|\sum_{k=2}^\infty \frac{z^{k-2}}{k!}\E(\sum_m\sigma'_m|Z_{n,m}|^2)^k\right|\le 2\epsilon^{-2}e^{A_\epsilon \epsilon^2}.$$
Taking this into \eqref{eq:taylor_exp_spheric}, we get
$$\left|\E e^{z\sum_m\sigma'_m|Z_{n,m}|^2}-1\right|\le 2\epsilon^{-2}e^{A_\epsilon \epsilon^2}|z|^2$$
for any $|z|\le \varepsilon/2$. Using the inequality $|\log(1+z)|\le K|z|$ for $|z|\le \epsilon<1$, we conclude that there exists $\varepsilon>0$, $A>0$ such that when $|z|<\varepsilon$, we have
$$|\Phi_M(z)|=|\log\E e^{z\sum_m\sigma'_m|Z_{n,m}|^2}|\le A|z|^2.$$
\end{proof}
The remaining proof for spherical case is identical to the proof for Gaussian case from \eqref{eq:Pc_1_Atau2cn2_Gau} onward.

\subsubsection{Real case} In the real case, the proof is similar, so we omit the detail. To complete the proof, we only need to replace the corresponding items with the following mentioned properties in the proof of complex case. 

The first, when $\Cb_N$, $\Zb$, $\Rb_M$ are all real, one has
$$\hUpsilon_M(\theta)=\Re(\hUpsilon(\theta))=\frac{1}{N}\sum_{n=1}^N\zb_n\Re(\Qb_M(\theta))\zb_n^*.$$

The second, note also that
$$\tr(\Re\Qb_M(\theta))^2\le \tr\Qb_M^2(\theta)$$
and 
$$\log \E e^{z|G|^2}=-\frac{1}{2}\log (1-2z)$$
for standard real Gaussian variable $G$ with $|z|<1/2$. 

The third, if $\gb\in\R^M$ is a standard real Gaussian vector, we have
$$\frac{M^k}{\E\|\gb\|^{2k}}=\frac{M^k}{M(M+2)\cdots(M+2k-2)}\le 1.$$

\subsection{Relative error bound for all $\theta$ by discretization.} \label{subsec:discret_unbias}
Let $\beta$ be a positive integer to be determined afterwards. For $k=0,\dots,M^\beta$, let 
$$\theta_k:=\frac{2\pi k}{M^\beta}.$$
For $\theta\in (0,2\pi)$, let $\theta_j$ be such that $\theta_{j-1}<\theta\le \theta_j$ if $\theta\in(0,\pi]$, and $\theta_{j}\le \theta < \theta_{j+1}$ if $\theta\in (\pi,2\pi)$. We write 
$$\begin{aligned}\frac{|\hUpsilon_M(\theta)-\Upsilon_M(\theta)|}{f(\theta)} \le & \frac{|\hat\Upsilon_M(\theta)-\hat\Upsilon_M(\theta_j)|}{f(\theta)} + \frac{|\hat\Upsilon_M(\theta_j)-\Upsilon_M(\theta_j)|}{f(\theta)} + \frac{|\Upsilon_M(\theta)-\Upsilon_M(\theta_j)|}{f(\theta)} \\
=: & \chi_1(\theta)+\chi_2(\theta)+\chi_3(\theta).\end{aligned}$$

We first estimate the probability of large deviation of $\chi_2$, which we will see is the main part of \eqref{eq:main_result_large_dev}. Different from the SRD case, we have to treat the singularity of the spectral density $f$ at $0$. We note that
$$\chi_2(\theta)=\chi_2(\theta_j)\frac{f(\theta_j)}{f(\theta)}.$$
We now prove that $\frac{f(\theta_j)}{f(\theta)}$ is bounded for $\theta\in [0,2\pi]$ and $\theta_j$. Because $f$ is supposed to be even and $2\pi$-periodic, we only need to consider $\theta\in (0,\pi)$. Note that by Lemma~\ref{lem:density_property}\ref{item:density_int_0} below, as $\theta_j\to 0^+$,
$$\frac{f(\theta_j)}{f(\theta)}\le \frac{f(\theta_j)}{\inf_{0<t\le \theta_j}f(t)}\sim 1.$$
Let $\delta>0$ be such that 
$$\frac{f(\theta_j)}{f(\theta)}\le \frac{f(\theta_j)}{\inf_{0<t\le \theta_j}f(t)}\le 2$$
for $0<\theta\le \theta_j\le \delta$. Then for any $\theta\in (0,\pi)$, we have
$$\frac{f(\theta_j)}{f(\theta)}\le \max\left(2, \frac{\sup_{t\in [\delta,\pi]}f(t)}{\inf_{t\in [\delta,\pi]}f(t)}\right),$$
and by \ref{ass:spectral_density}, \ref{ass:spectral_bound_below}, the RHS of the above inequality is bounded. Denote this bound as $F$. Using Proposition~\ref{prop:indiv_theta_dev}, for any $x\in (0,CFK_1/\kappa)$, as $M,N$ are large enough, we have
$$\P\left(\sup_{\theta\in(\theta_{j-1},\theta_j]}\chi_2(\theta) > x \right) \le \P\left(\chi_2(\theta_j) > \frac{x}{F} \right)
    \le 2\exp\left(-\frac{Nx^2}{CF^2K_2\log^2 M}\right).$$
Then 
\begin{equation}
    \P\left(\sup_{\theta\in(0,2\pi)}\chi_2(\theta) > x \right)
    \le 2M^\beta\exp\left(-\frac{Nx^2}{CF^2K_2\log^2 M}\right). \label{eq:bound_chi2}
\end{equation}

We then estimate $\chi_1$. From the proof of Lemma~10 in \cite{vinogradova2015estimation}, and note that $f(\theta)$ is bounded away from $0$, also note Lemma~\ref{lem:norm_toeplitz} for the bound of $\|\Rb_M\|$, and \ref{ass:C_bound} for the bound of $\|\Cb_N\|$, we have
$$\begin{aligned}\sup_{\theta\in[0,2\pi]}\chi_1(\theta) & \le \sup_{\theta\in[0,2\pi]}\frac{1}{Nf(\theta)}\|\Cb_N\|\|\Qb_M(\theta)-\Qb_M(\theta_j)\| |\theta-\theta_j|\sum_{m,n}|Z_{n,m}|^2 \\
    &\lesssim \sup_{\theta\in[0,2\pi]}\frac{1}{N}\|\Cb_N\|\|\Rb_M\|M\sqrt{\log M} |\theta-\theta_j| \sum_{m,n}|Z_{n,m}|^2 \\
    &\lesssim \frac{\kappa}{N}M^{1+a-\beta}L(M)(\log M)^{3/2} \sum_{m,n}|Z_{n,m}|^2 \\
    &\le \kappa M^{2+a-\beta}L(M)(\log M)^{3/2}\frac{\sum_{m,n}|Z_{n,m}|^2}{MN}. \end{aligned}$$
Then for $x_M$ satisfying $x_M\gtrsim M^{-\gamma}$ with some $\gamma>0$, we have
\begin{equation}
    \P\left(\sup_{\theta\in[0,2\pi]} \chi_1(\theta)> x_M\right) \le \P\left(\frac{\sum_{m,n}|Z_{n,m}|^2}{MN}>\frac{M^{\beta-2-a-\gamma}}{\kappa L(M)(\log(M))^{3/2}}\right). \label{eq:bound_chi1_0}
\end{equation}
We take $\beta > 2+a+\gamma$ and let $\varepsilon=\frac{\beta-2-a-\gamma}{2}$, then as $M$ is large enough, we have 
\begin{equation}
    \frac{M^{\beta-2-a-\gamma}}{\kappa L(M)(\log(M))^{3/2}} \gg M^\varepsilon>1. \label{eq:M_beta_w_a_gg_M_varepsilon}
\end{equation} 
If $\zb_n$ are standard complex normal, by \cite[Lemma~2]{vinogradova2015estimation}, we have for any $y>1$,
\begin{equation}
    \P\left(\frac{\sum_{m,n}|Z_{n,m}|^2}{MN}>y\right)\le \exp(-MN(y-1-\log y)). \label{eq:P_sum_Z_MN_le_MNy}
\end{equation}
For $y$ large enough, we have $y-1-\log y>y/2$. Thus if $M$ is large enough, from \eqref{eq:bound_chi1_0}, \eqref{eq:M_beta_w_a_gg_M_varepsilon} and \eqref{eq:P_sum_Z_MN_le_MNy}, 
\begin{equation}
    \P\left(\sup_{\theta\in[0,2\pi]} \chi_1(\theta)> x_M\right) \le \exp(-NM^{1+\varepsilon}/2). \label{eq:bound_chi1}
\end{equation} 
The real Gaussian case is similar.

If $\zb_n$ are spherical of radius $\sqrt{M}$, then $\frac{\sum_{m,n}|Z_{n,m}|^2}{MN}=1$ and the RHS of \eqref{eq:bound_chi1_0} is eventually zero. Therefore, for both spherical and Gaussian cases, when $M$ is large enough, \eqref{eq:bound_chi1} holds.

We now estimate the bound of $\chi_3$. From the proof of Lemma~12 in \cite{vinogradova2015estimation}, and note that $|\xi_N|\le \sqrt{\tr\Cb_N^2/N}\le \sqrt{C}$ we have
$$\begin{aligned}\sup_{\theta\in[0,2\pi]}\chi_3(\theta) &\lesssim M^2|\theta-\theta_j|\|\Rb_M\|\sqrt{C\log M}\\
    &\lesssim M^{2+a-\beta}L(M)\sqrt{C\log M}. \end{aligned}$$
For any $x_M$ satisfying $x_M\gtrsim M^{-\gamma}$, let $\beta>2+a+\gamma$, then as $M$ is large enough, we have
    $$\sup_{\theta\in[0,2\pi]}\chi_3(\theta) < x_M.$$
and thus 
\begin{equation}
    \P\left(\sup_{\theta\in[0,2\pi]}\chi_3(\theta)>x_M\right)=0. \label{eq:bound_chi3}
\end{equation}

The final result follows from combining the above estimations \eqref{eq:bound_chi2}, \eqref{eq:bound_chi1}, \eqref{eq:bound_chi3}, and letting the dominant item \eqref{eq:bound_chi2} absorb the others by appropriately changing the corresponding constants.

\subsection{Proof of Proposition~\ref{prop:traceQ_bound_0}}
\label{subsec:proof_estimation_trQb2}
In order to estimate $\tr\Qb_M^2(\theta)$, we first estimate the norm of the Toeplitz matrix $\Rb_M$. The following lemma is a direct corollary of Theorem~2.3 in \cite{tian2022joint}, so the proof is omitted. 

\begin{lemma} \label{lem:norm_toeplitz} If $(\Rb_M)$ is a sequence of Toeplitz matrices satisfying \ref{ass:spectral_density} and \ref{ass:spec_den_behav_0}, then
$$\|\Rb_M\|\asymp M^aL(M).$$
\end{lemma}

We also need the following properties of functions regularly varying at $0$.
\begin{lemma} \label{lem:density_property} If $f$ satisfies \ref{ass:spec_den_behav_0}, then
\begin{enumerate}[label=(\alph*)]
    \item \label{item:density_sup} $\sup \{f(t) \tq x\le t\le \pi\}\sim f(x)$ as $x\to 0^+$.
    \item \label{item:density_inf} $\inf \{f(t) \tq 0< t\le x\}\sim f(x)$ as $x\to 0^+$.
    \item \label{item:density_int_0} $\int_0^x f(t)\dd t \sim \frac{x^{1-a}L(x^{-1})}{1-a}=\frac{xf(x)}{1-a}$ as $x\to 0^+$.
\end{enumerate}
\end{lemma}
\begin{proof}
By changing the variable $u=x^{-1}$,  \ref{item:density_sup} and \ref{item:density_inf} follow from Theorem~1.5.3 of \cite{bingham1989regular}, and \ref{item:density_int_0} from Proposition~1.5.10 of \cite{bingham1989regular}.
\end{proof}

The bound of $\tr\Qb_M^2(\theta)$ will be estimated in different ways according to the location of $\theta$.  From Lemma~\ref{lem:norm_toeplitz} above, and (11) in \cite{vinogradova2015estimation}, we get a global estimation 
\begin{equation}
    \tr\Qb_M^2(\theta)=O(M^{2a}L^2(M)\log M) \label{eq:trQ_bound_trival}
\end{equation} 
for any $\theta\in [-\pi,\pi]$. This bound may be sharp for $\theta$ very close to the singular point $0$, but not for $\theta$ farther away from $0$.

In order to establish a sharper bound of $\tr\Qb_M^2(\theta)$ for a regular point $\theta$, we define, for a certain $\delta>0$, a local $\infty$-norm $\|f\|_{(\theta,\delta)}$ as
\begin{equation}
    \|f\|_{(\theta,\delta)}:=\esssup_{t\in (\theta-\delta,\theta+\delta)}\{|f(t)|\}. \label{eq:def_local_norm}
\end{equation}

\begin{prop} \label{prop:bound_trQ2} Let $\Qb_M(\theta)$ be defined in \eqref{eq:def_QT} with $\Rb_M$ having positive spectral density $f\in L^1(-\pi,\pi)$. Then there exists an absolute constant $K>0$ such that for any $\theta\in\R$ and $\delta\in (0,\pi/2)$, 
\begin{equation}
    \frac{\tr \Qb_M^2(\theta)}{2\log M}\le \|f\|_{(\theta,\delta)}^2 + \frac{K\|f\|_1(\|f\|_1+\|f\|_{(\theta,\delta)})}{\delta^2\log M}. \label{eq:bound_trQ2}
\end{equation}
\end{prop}

Before proving this proposition, it is convenient to express $\tr\Qb_M^2(\theta)$ in terms of $f$. Using the integral expression of entries $r_{i-j}$, we write
\begin{equation}\begin{aligned} &\tr \Qb_M^2(\theta) = \sum_{i,j,k,l}r_{i-j}\frac{e^{-\im (j-k)\theta}}{M-|j-k|}r_{k-l}\frac{e^{-\im(l-i)\theta}}{M-|l-i|} \\
    =&\frac{1}{4\pi^2}\sum_{i,j,k,l}\int_{-\pi}^\pi f(x)e^{-\im (i-j)x}\dd x\frac{e^{-\im (j-k)\theta}}{M-|j-k|}\int_{-\pi}^\pi f(y)e^{-\im (k-l)y}\dd y\frac{e^{-\im(l-i)\theta}}{M-|l-i|} \\
    =&\frac{1}{4\pi^2}\int_{-\pi}^{\pi}\int_{-\pi}^{\pi}f(x+\theta)f(y+\theta)\sum_{i,j,k,l}\frac{e^{\im(jx-ky)-\im(ix-ly)}}{(M-|j-k|)(M-|l-i|)}\dd x\dd y \\
    =&\frac{1}{4\pi^2}\int_{-\pi}^{\pi}\int_{-\pi}^{\pi}f(x+\theta)f(y+\theta)\left|\sum_{i,j}\frac{e^{\im(jx-iy)}}{M-|i-j|}\right|^2\dd x\dd y\,.\end{aligned} \label{eq:trQ2_int}
\end{equation}
We denote
$$g(x,y):=\sum_{i,j}\frac{e^{\im(jx-iy)}}{M-|i-j|}.$$
For later use, it is necessary to study several bounds of this integral kernel $g(x,y)$.

\begin{lemma}\label{lem:kernel_bounds}The kernel $g(x,y)=\sum_{0\le i,j\le M-1}\frac{e^{\im(jx-iy)}}{M-|i-j|}$ satisfies
\begin{enumerate}[label={(\arabic{enumi})}]
    \item \label{enum:g2_density_log} $\frac{1}{4\pi^2}\int_{-\pi}^\pi\int_{-\pi}^\pi |g(x,y)|^2\dd x\dd y=1+2\sum_{k=1}^{M-1}\frac{1}{k}\sim 2\log M$.
    \item \label{enum:g_1var_int}$\frac{1}{2\pi}\int_{-\pi}^\pi |g(x,y)|^2\dd x\lesssim \min(\frac{\log^2 M}{|\sin(y/2)|},\frac{1}{\sin^2(y/2)})$.
    \item \label{enum:g_2var_bound} For any $\delta\in(0,\pi)$ and $x,y\in [-\pi,\pi]\backslash(-\delta,\delta)$, we have
    $$|g(x,y)|\lesssim \frac{1}{|\sin \delta|}.$$
\end{enumerate}
\end{lemma}
\begin{proof}
For \ref{enum:g2_density_log}, we replace $f$ with $1$ and $\Rb_M$ with $\Ib$ correspondingly in eq.\eqref{eq:trQ2_int}, and we get
$$\frac{1}{4\pi^2}\int_{-\pi}^\pi\int_{-\pi}^\pi |g(x,y)|^2\dd x\dd y=\tr\Bb_M^2=1+2\sum_{k=1}^{M-1}\frac{1}{k}.$$

For \ref{enum:g_1var_int}, note that 
$$\frac{1}{2\pi}\int_{-\pi}^\pi |g(x,y)|^2\dd x=\sum_j\left|\sum_i\frac{e^{-\im iy}}{M-|i-j|}\right|^2.$$
Multiplying $|1-e^{-\im y}|=2|\sin(y/2)|$ or $|1-e^{-\im y}|^2=4\sin^2(y/2)$ on both sides, and note that
$$\begin{aligned} & \left|\sum_{i=0}^{M-1}\frac{e^{-\im iy}(1-e^{-\im y})}{M-|i-j|}\right| \\
&= \left|\frac{1}{M-j}-\frac{e^{-\im My}}{M-|M-1-j|}+\sum_{i=1}^{M-1}\frac{(|i-j|-|i-j-1|)e^{-\im iy}}{(M-|i-j|)(M-|i-j-1|)}\right| \\
&\le \frac{1}{M-j}+\frac{1}{M-|M-1-j|}+\sum_{i=1}^{M-1}\frac{1}{(M-|i-j|)(M-|i-j-1|)} \\
    &\lesssim \frac{1}{M-j}+\frac{1}{M-|M-1-j|},\end{aligned}$$
we get
$$|\sin(y/2)|\int_{-\pi}^\pi |g(x,y)|^2\dd x\lesssim \log(M)\sum_{j=0}^{M-1}\left(\frac{1}{M-j}+\frac{1}{M-|M-1-j|}\right)\lesssim\log^2 M,$$
and
$$\sin^2(y/2)\int_{-\pi}^\pi |g(x,y)|^2\dd x\lesssim \sum_{j=0}^{M-1}\left(\frac{1}{M-j}+\frac{1}{M-|M-1-j|}\right)^2\le K,$$
and \ref{enum:g_1var_int} follows.

For \ref{enum:g_2var_bound}, we rewrite $g(x,y)$ as
\begin{equation}
    \begin{aligned} g(x,y) =& \frac{1}{M}\sum_{j=0}^{M-1}e^{\im j(x-y)}+ \sum_{m=1}^{M-1}\frac{\sum_{i=0}^{M-m-1}e^{\im (ix+mx-iy)} + \sum_{j=0}^{M-m-1}e^{\im (jx-my-jy)}}{M-m}\\
    =& \frac{1}{M}\frac{1-e^{\im M(x-y)}}{1-e^{\im (x-y)}}  + \frac{e^{\im Mx}}{1-e^{\im (x-y)}}\sum_{m=1}^{M-1}\frac{e^{-\im(M-m)x}-e^{-\im(M-m)y}}{M-m} \\
    & + \frac{e^{-\im My}}{1-e^{\im (x-y)}}\sum_{m=1}^{M-1}\frac{e^{\im(M-m)y}-e^{\im(M-m)x}}{M-m} \\
    =& \frac{1}{M}\frac{1-e^{\im M(x-y)}}{1-e^{\im (x-y)}}+ \frac{e^{\im Mx}}{1-e^{\im (x-y)}}\sum_{k=1}^{M-1}\frac{e^{-\im kx}-e^{-\im ky}}{k} \\
    & + \frac{e^{-\im My}}{1-e^{\im (x-y)}}\sum_{k=1}^{M-1}\frac{e^{\im ky}-e^{\im kx}}{k} \\
    =:& g_1(x,y)+g_2(x,y)+g_3(x,y).
    \end{aligned} \label{eq:expr_g} \end{equation} 
We note that for any $x,y\in \R$, 
$$|g_1(x,y)|=\frac{1}{M}\left|\frac{\sin (M(x-y)/2)}{\sin((x-y)/2)}\right|\le 1.$$
Let $z_1=e^{-\im x}, z_2=e^{-\im y}$, then we have
$$|g_2(x,y)|= \left|\frac{1}{z_1-z_2}\int_{[z_1,z_2]} \sum_{k=0}^{M-2} z^k\dd z\right|\le \sup_{z\in[z_1,z_2]}\left|\frac{1-z^{M-1}}{1-z}\right|\le \frac{1}{|\sin \delta|},$$
where $[z_1,z_2]$ denotes the segment between $z_1$ and $z_2$. The same bound also controls $g_3$, so \ref{enum:g_2var_bound} holds.
\end{proof}

Now we are ready to prove Proposition~\ref{prop:bound_trQ2}.

\begin{proof}[Proof of Proposition~\ref{prop:bound_trQ2}] For $\delta\in (0,\pi/2)$, let $E_\delta:=[-\pi,\pi]\backslash (-\delta,\delta)$. Then we have
\begin{equation}
    \begin{aligned} & \tr\Qb_M^2(\theta) = \frac{1}{4\pi^2}\int_{-\pi}^{\pi}\int_{-\pi}^{\pi}f(x+\theta)f(y+\theta)|g(x,y)|^2\dd x\dd y \\
    = & \frac{1}{4\pi^2}\left(\int_{-\delta}^{\delta}\int_{-\delta}^{\delta}+\int_{E_\delta}\int_{E_\delta}+\int_{-\delta}^{\delta}\int_{E_\delta}+\int_{E_\delta}\int_{-\delta}^{\delta}\right)f(x+\theta)f(y+\theta)|g(x,y)|^2\dd x\dd y \\
    =: & P_1+P_2+P_3+P_4. \end{aligned} \label{eq:trQ2_integral}
\end{equation}
For $P_1$, using Lemma~\ref{lem:kernel_bounds} \ref{enum:g2_density_log}, we have
$$|P_1|\le \|f\|_{\theta,\delta}\frac{1}{4\pi^2}\int_{-\delta}^{\delta}\int_{-\delta}^{\delta}|g(x,y)|^2\dd x\dd y\le 2\|f\|_{\theta,\delta}\log M.$$
For $P_2$, using Lemma~\ref{lem:kernel_bounds} \ref{enum:g_2var_bound}, we have
$$|P_2|\lesssim \frac{1}{\sin ^2\delta}\int_{E_\delta}\int_{E_\delta}f(x+\theta)f(y+\theta)\dd x\dd y\lesssim \frac{\|f\|_1^2}{\delta^2}.$$
For $P_3$ and similarly for $P_4$, using Lemma~\ref{lem:kernel_bounds} \ref{enum:g_1var_int}, we have
$$|P_3| \le \|f\|_{(\theta,\delta)}\frac{1}{4\pi^2} \int_{E_\delta}f(y+\theta)\int_{-\delta}^{\delta}|g(x,y)|^2\dd x\dd y\lesssim  \frac{\|f\|_{(\theta,\delta)}\|f\|_1}{\delta^2},$$
and $|P_4|$ is controlled by the same bound.

Summing up the bounds for $P_1,P_2,P_3,P_4$ and dividing $8\pi^2\log M$, the result follows.
\end{proof}

As a consequence of Proposition~\ref{prop:bound_trQ2}, if $f$ is bounded in a neighborhood of a point or a set, then $\tr \Qb_M^2(\theta)/\log M$ is (uniformly) bounded at this point or in this set.

To summarize what we have obtained, if $f$ satisfies \ref{ass:spectral_density}, \ref{ass:spectral_bound_below} and \ref{ass:spec_den_behav_0}, then
\begin{equation}
    \frac{\tr\Qb_M^2(\theta)}{f^2(\theta)\log M} \label{eq:trQ2_log}
\end{equation}
is bounded uniformly in $M\ge 1$ and
\begin{enumerate}
    \item in $\theta\in [-\frac{A}{M},\frac{A}{M}]$ for any $A>0$, using the global bound \eqref{eq:trQ_bound_trival} and Lemma~\ref{lem:density_property}\ref{item:density_inf}.
    \item in $\theta\in [-\pi,-\delta]\cup [\delta,\pi]$ for any $\delta\in (0,\pi/2)$, using Proposition~\ref{prop:bound_trQ2}.
\end{enumerate}

Therefore, using a classic argument, we can find two sequences of positive numbers $1/M \ll \tau_M < \delta_M \ll 1$ such that \eqref{eq:trQ2_log2} is uniformly bounded in $[-\pi,-\delta_M]\cup [-\tau_M,\tau_M]\cup [\delta_M,\pi]$. In order to complete the proof of Proposition~\ref{prop:traceQ_bound_0}, it remains to prove the uniform boundedness of \eqref{eq:trQ2_log2} for $|\theta|\in (\tau_M,\delta_M)$. 

For any such $\theta$, suppose that $\theta>0$ without loss of generality, and denote $E_\theta:=[-\pi,\pi]\backslash(-\frac{\theta}{2},\frac{\theta}{2})$. We write
\begin{equation}
    \begin{aligned} \tr\Qb_M^2(\theta) &= \frac{1}{4\pi^2}\int_{-\pi}^{\pi}\int_{-\pi}^{\pi}f(x)f(y)|g(x-\theta,y-\theta)|^2\dd x\dd y \\
    &=: P_1+P_2+P_3+P_4, \end{aligned} \label{eq:trQ2_integral_theta_near_0}
\end{equation}
where the $P_i$'s are the four sub-integrals defined via the following partition of  the double integral
\[        \int_{-\pi}^{\pi}\int_{-\pi}^{\pi} = \int_{E_\theta}\int_{E_\theta}+\int_{-\frac{\theta}{2} }^{\frac{\theta}{2}}\int_{-\frac{\theta}{2} }^{\frac{\theta}{2}}+\int_{-\frac{\theta}{2} }^{\frac{\theta}{2}}\int_{E_\theta}+\int_{E_\theta}\int_{-\frac{\theta}{2}}^{\frac{\theta}{2}}.
\] 
For $P_1$, by Lemma~\ref{lem:density_property} \ref{item:density_sup} and the definition of regularly varying functions, with some $\varepsilon>0$, whenever $\theta$ is small enough,
\begin{equation}
    \sup_{x\in E_\theta}f(x) \le (1+\varepsilon) f(\theta/2) \le 2^a(1+\varepsilon)^2 f(\theta). \label{eq:sup_E_theta_f}
\end{equation}
Thus by Lemma~\ref{lem:kernel_bounds}\ref{enum:g2_density_log},
$$|P_1|\lesssim \frac{f^2(\theta)}{4\pi^2}\int_{E_\theta}\int_{E_\theta}|g(x-\theta,y-\theta)|^2\dd x\dd y \lesssim f^2(\theta)\log M.$$

For $P_2$, by Lemma~\ref{lem:kernel_bounds}\ref{enum:g_2var_bound} and Lemma~\ref{lem:density_property}\ref{item:density_int_0}, we have
$$|P_2|\lesssim \sup_{x,y\in [-\frac{3\theta}{2},-\frac{\theta}{2}]}|g(x,y)|^2\left(\int_{-\frac{\theta}{2}}^\frac{\theta}{2}f(x)\dd x\right)^2\lesssim f^2(\theta).$$

For $P_3$, and similarly for $P_4$, using \eqref{eq:sup_E_theta_f} again, and using Lemma~\ref{lem:kernel_bounds}\ref{enum:g_1var_int} and Lemma~\ref{lem:density_property}\ref{item:density_int_0}, we have
$$|P_3|\lesssim f(\theta)\int_{-\frac{\theta}{2} }^{\frac{\theta}{2}}f(y)\int_{-\pi}^\pi|g(x,y)|^2\dd x\dd y\lesssim f^2(\theta)\log^2 M.$$
The same bound also controls $P_4$. 

The proof of Proposition~\ref{prop:traceQ_bound_0} is complete by summing up the bounds for $P_1, P_2, P_3, P_4$.

\section{Proof of Proposition~\ref{prop:inconsis_Rb}} \label{sec:proof_inconsis_Rb}
Define
$$r_k^b:=\left(1-\frac{|k|}{M}\right)r_k,$$
and
$$\hUpsilon^b_M(\theta):=\sum_{k=-M+1}^{M-1}\hat r_k^b e^{\im k\theta}, \quad \Upsilon^b_M(\theta):=\sum_{k=-M+1}^{M-1} r_k^b e^{\im k\theta}.$$
Note that $\Upsilon^b_M$ is the Ces\`aro mean of $\Upsilon_M(\theta):=\sum_{k=-M+1}^{M-1} r_k e^{\im k\theta}$, therefore
\begin{equation}
    \Upsilon^b_M(\theta)=\frac{1}{2\pi}\int_{-\pi}^{\pi} f(x)F_M(\theta-x)\dd x,
\end{equation}
where $F_M(x)=\frac{\sin^2(Mx/2)}{M\sin^2(x/2)}$ is the Fej\'er kernel. Thus for any $\theta\in\R$, we have
\begin{equation}
    \essinf_{t}f(t)\le \Upsilon^b_M(\theta)\le \esssup_{t} f(t). \label{eq:range_upsilonb}
\end{equation}
By \ref{ass:spectral_bound_below}, $\Upsilon^b_M$ is positive and uniformly lower bounded from $0$.

Following the same idea as \S\ref{subsec:preliminary_pr_main}, we only need to estimate 
\begin{equation}
     \P\left(\sup_{\theta\in[0,2\pi]}\left|\frac{\hUpsilon^b_M(\theta)}{\Upsilon^b_M(\theta)}-1\right|>x\right) \label{eq:to_ctrl_prob_bias}
\end{equation}
for any $x>0$. 

We use the same discretization strategy as \S\ref{sec:proof_main_th}. For any fixed $\theta\in [0,2\pi]$, define
$$d_M(\theta)=\frac{1}{\sqrt{M}}(1,e^{-\im \theta},\dots, e^{-\im (M-1)\theta})^\tran,\quad \Qb_M^b(\theta):= \Rb_M^{1/2}d_M(\theta)d_M^*(\theta)\Rb_M^{1/2}$$
then by Lemma~3 of \cite{vinogradova2015estimation}, we have
$$\hUpsilon^b_M(\theta)=\frac{1}{N}d_M^*(\theta)\Rb_M^{1/2}\Zb^*\Cb_N\Zb\Rb_M^{1/2}d_M(\theta)=\frac{1}{N}\sum_{n=1}^N c_n\zb_n\Qb_M^b(\theta)\zb_n^*,$$
and by the unitary invariance of $\zb_n$, note also that $\Qb_M^b(\theta)$ is of rank one with a positive eigenvalue $d_M^*(\theta)\Rb_M d_M(\theta)=\Upsilon^b_M(\theta)$, we then have
$$\frac{\hUpsilon^b_M(\theta)}{\Upsilon^b_M(\theta)}\eqinlaw \frac{1}{N}\sum_{n=1}^N c_n|Z_{n,1}|^2.$$
Then
$$\P\left(\left|\frac{\hUpsilon^b_M(\theta)}{\Upsilon^b_M(\theta)}-1\right|>x\right)=\P\left(\left|\frac{1}{N}\sum_{n=1}^N c_n(|Z_{n,1}|^2-1)\right|>x\right).$$
Using the same method as \S\ref{subsec:large_dev_ind_theta}, we get the concentration inequality
$$\P\left(\left|\frac{1}{N}\sum_{n=1}^N c_n(|Z_{n,1}|^2-1)\right|>x\right)\le 2\exp\left(-\frac{KNx^2}{\kappa^2\log^2 M}\right)$$
for some constant $K>0$, for any $x>0$ and $M,N$ large enough.

For the discretization step, we use the same method as \S\ref{subsec:discret_unbias}, and the proof of Lemma~4, Lemma~6 in \cite{vinogradova2015estimation}, along with the norm bound $\|\Rb_M\|\lesssim M^aL(M)$. Note also that $\Upsilon_M^b(\theta)$ are positive and uniformly lower bounded from $0$. We finally get
$$\P\left(\esssup_{\theta}\left|\frac{\hUpsilon^b_M(\theta)}{\Upsilon^b_M(\theta)}-1\right|>x\right)\le 2M^\beta\exp\left(-\frac{KNx^2}{\kappa^2\log^2 M}\right)$$
for some $\beta>0$, $K>0$ and any $x>0$, large enough $M,N$. This implies that 
$$\|(\Rb_M^b)^{-1/2}\hRb^b_M(\Rb_M^b)^{-1/2}-\Ib\|\to 0$$
as $M,N\to\infty$ with $N\gg \log^3 M$.

Next we prove the inconsistency \eqref{eq:inconsistency}. We first prove that for two sequences of invertible matrices $\Rb_{1,M}, \Rb_{2,M}$, a necessary condition for the convergence
\begin{equation}
    \|\Rb_{1,M}^{-1/2}\Rb_{2,M}\Rb_{1,M}^{-1/2}-\Ib\|\xrightarrow[M\to\infty]{} 0\,, \label{eq:assumpt_absurd}
\end{equation}
is
\begin{equation}
    \lim_{M\to\infty} \frac{\lmax(\Rb_{2,M})}{\lmax(\Rb_{1,M})}=1. \label{eq:max_lim_gamma}
\end{equation}
Take an arbitrary $\varepsilon>0$. Let $u$ be un eigenvector of $\Rb_{1,M}$ associated with $\lmax(\Rb_{1,M})$, then from \eqref{eq:assumpt_absurd}, for large enough $M$,
\begin{equation}
    1-\varepsilon<u^*(\Rb_{1,M})^{-1/2}\Rb_{2,M}(\Rb_{1,M})^{-1/2}u=\frac{u^*\Rb_{2,M} u}{\lmax(\Rb_{1,M})}\le \frac{\lmax(\Rb_{2,M})}{\lmax(\Rb_{1,M})}. \label{eq:1ep_le_lmax}
\end{equation}
Note that $(\Rb_{1,M}^{-1/2}\Rb_{2,M}\Rb_{1,M}^{-1/2})^{-1}$ has the same eigenvalues as $\Rb_{2,M}^{-1/2}\Rb_{1,M}\Rb_{2,M}^{-1/2}$. Recall that for a sequence of Hermitian matrices $\Ab_M$, the convergence $\|\Ab_M-\Ib\|\to 0$ is equivalent to the convergence of its eigenvalues, i.e. $\lmax(\Ab_M)\to 1$, $\lmin(\Ab_M)\to 1$. Therefore \eqref{eq:assumpt_absurd} also implies that
$$\|(\Rb_{2,M})^{-1/2}(\Rb_{1,M})(\Rb_{2,M})^{-1/2}-\Ib\|\xrightarrow[M\to\infty]{} 0\,.$$
Using the same arguments as \eqref{eq:1ep_le_lmax}, we get, for large enough $M$,
\begin{equation}
    1-\varepsilon \le \frac{\lmax(\Rb_{1,M})}{\lmax(\Rb_{2,M})}. \label{eq:lmin_lmax_1ep}
\end{equation}
Combining \eqref{eq:1ep_le_lmax} and \eqref{eq:lmin_lmax_1ep}, we have
\begin{equation}
    \lim_{M\to\infty}\frac{\lmax(\Rb_{2,M})}{\lmax(\Rb_{1,M})}= 1\,, \label{eq:lambda_ratio_1}
\end{equation}
and \eqref{eq:max_lim_gamma} follows. However, we will prove that almost surely \eqref{eq:max_lim_gamma} cannot be satisfied by $\hRb_M^b$ and $\Rb_M$. Indeed, from \eqref{eq:ratio_consist_hRbRb} we conclude that almost surely 
$$\frac{\lmax(\hRb_M^b)}{\lmax(\Rb_M^b)}\to 1.$$
Thus we only need to prove that
\begin{equation}
    \frac{\lmax(\Rb_M^b)}{\lmax(\Rb_M)}\not\to 1. \label{eq:lmax_lim_not1}
\end{equation}
Let $\K$ and $\K^b$ be two integral operators acting on $L^2(0,1)$ defined by
$$\K(\varphi)(x)=\int_0^1\frac{1}{|x-y|^{1-a}}\varphi(y)\dd y,\quad \K^b(\varphi)(x)=\int_0^1\frac{1-|x-y|}{|x-y|^{1-a}}\varphi(y)\dd y.$$

\begin{lemma} Under the same assumptions as Proposition~\ref{prop:inconsis_Rb}, as $M\to\infty$, 
$$\frac{\lmax(\Rb_M)}{K M^aL(M)}\to \lambda_1(\K), \quad \frac{\lmax(\Rb^b_M)}{K M^aL(M)}\to \lambda_1(\K^b)$$
with some absolute constant $K>0$.
\end{lemma}
\begin{proof}
We first assume that the slowly varying function $L$ in \ref{ass:spec_den_behav_0} equals to $1$. Then from \cite[Proposition~2.2.16]{pipiras2017long},
$$r_k\sim \frac{K}{(1+|k|)^{1-a}}$$
as $k\to\infty$ with some absolute constant $K>0$. Using Widom-Shampine's Lemma (\cite[Lemma~5.1]{merlevede2019unbounded}) and the same method as the proof of \cite[Theorem~2.3]{merlevede2019unbounded}, one can prove that
\begin{equation}
    \frac{\lmax(\Rb_M)}{K M^a}\to \lambda_1(\K), \quad \frac{\lmax(\Rb^b_M)}{K M^a}\to \lambda_1(\K^b). \label{eq:cv_normalized_eigenvalues}
\end{equation}

If the slowly varying function $L$ in \ref{ass:spec_den_behav_0} is not constant, let 
$$\tilde f(\theta)=\frac{1}{|\theta|^a}, \quad \theta \in [-\pi,\pi]$$
and $\tilde \Upsilon^b_M$, $\tilde \Rb_M^b$ be defined with $\tilde f$ in the same way as $\Upsilon^b_M$, $\Rb_M^b$ with $f$. Note that the F\'ejer kernel $F_M$ has the same upper bound as the Dirichlet kernel $D_M(\theta)=\sin((M+1/2)\theta)/\sin(\theta/2)$ used in the proof of \cite[Theorem~2.3]{tian2022joint}, that is, for $\theta\in [-3\pi/2, 3\pi/2]$, 
$$|F_M(\theta)|=\left|\frac{D_0(\theta)+\cdots+D_{M-1}(\theta)}{M}\right|\lesssim \min\{M, |\theta|^{-1}\}.$$
Then using the same technique there, one can prove that
$$\sup_{\theta}\left\|\frac{\Upsilon^b_M(\theta)}{M^aL(M)}-\frac{\tilde\Upsilon^b_M(\theta)}{M^a}\right\|\to 0$$
as $M\to\infty$, which implies that
$$\left\|\frac{\Rb_M^b}{M^aL(M)}-\frac{\tilde\Rb_M^b}{M^a}\right\|\to 0.$$
Also note that by Theorem~2.3 of \cite{tian2022joint},
$$\left\|\frac{\Rb_M}{M^aL(M)}-\frac{\tilde\Rb_M}{M^a}\right\|\to 0,$$
together with \eqref{eq:cv_normalized_eigenvalues}, we have
$$\frac{\lmax(\Rb_M)}{K M^aL(M)}\to \lambda_1(\K), \quad \frac{\lmax(\Rb^b_M)}{K M^aL(M)}\to \lambda_1(\K^b).$$
\end{proof}

From this lemma, we have
\begin{equation}
    \frac{\lmax(\Rb^b_M)}{\lmax(\Rb_M)}\to \frac{\lambda_1(\K^b)}{\lambda_1(\K)}. \label{eq:ratio_lmax}
\end{equation}
We then prove that $\lambda_1(\K)>\lambda_1(\K^b)$. Indeed because the two integral kernels are positive, from the mini-max formula for the largest eigenvalue, their eigenfunctions associated with the largest eigenvalue are positive in $[0,1]$. Let $\varphi^b$ be the eigenfunction of $\K^b$ associated with $\lambda_1(\K^b)$, then
$$\lambda_1(\K^b)=\langle \varphi^b, \K^b\varphi^b\rangle = \langle \varphi^b, \K\varphi^b\rangle- \int_0^1\int_0^1|x-y|^a\varphi^b(x)\varphi^b(y)\dd x\dd y < \lambda_1(\K),$$
from which we conclude that 
\begin{equation}
    \lim_{M\to\infty}\frac{\lmax(\Rb^b_M)}{\lmax(\Rb_M)} = \frac{\lambda_1(\K^b)}{\lambda_1(\K)}<1\,. \label{eq:lmax_less_1}
\end{equation}
This is the end of the proof of proposition since \eqref{eq:lmax_lim_not1} is proved.

\ifthenelse{\boolean{supplementary}}{The proof of other results is given in the supplementary material.}{}

\subsection*{Acknowledgment} 
The authors are financially supported by Department of Statistics and Actuarial Science of the University of Hong Kong. We also thank Professor Romain Couillet in University of Grenoble-Alpes for posing this interesting question and also for fruitful discussions.

\appendix
\section{Additional proofs} \label{sec:other_proofs}

The proof of Corollary~\ref{corol:conv_1/2power} is in \S\ref{subsec:proof_conv_1/2power}; the proof of Proposition~\ref{prop:lsd_consistent} is in \S\ref{subsec:proof_lsd_consist}; and the proof of Proposition~\ref{prop:consistency_pca_whitened} is in \S\ref{subsec:proof_pca_whitened}.

\subsection{Proof of Corollary~\ref{corol:conv_1/2power}}\label{subsec:proof_conv_1/2power}
We write
\begin{equation}
    \left\|\hRb_M^{1/2}\Rb_M^{-1/2}-\sqrt{\xi_N}\Ib\right\|\le \|\Rb_M^{1/4}\|\left\|\Rb_M^{-1/4}\hRb_M^{1/2}\Rb_M^{-1/4}-\sqrt{\xi_N}\Ib\right\|\|\Rb_M^{-1/4}\|, \label{eq:corol_1/2conv}
\end{equation}
where $\|\Rb_M^{-1/4}\|$ is bounded, and from Lemma~\ref{lem:norm_toeplitz}, $\|\Rb_M^{1/4}\|=O(M^{1/4+\varepsilon})$ with any $\varepsilon\in (0,1/8)$. 

The spectral norm $\left\|\Rb_M^{-1/4}\hRb_M^{1/2}\Rb_M^{-1/4}-\sqrt{\xi_N}\Ib\right\|$ equals to 
$$\max\left\{\left|\lmax(\Rb_M^{-1/4}\hRb_M^{1/2}\Rb_M^{-1/4})-\sqrt{\xi_N}\right|, \left|\lmin(\Rb_M^{-1/4}\hRb_M^{1/2}\Rb_M^{-1/4})-\sqrt{\xi_N}\right|\right\}.$$
The positive definite Hermitian matrix $\Rb_M^{-1/4}\hRb_M^{1/2}\Rb_M^{-1/4}$ has the same eigenvalues as $\hRb_M^{1/2}\Rb_M^{-1/2}$, so the latter matrix has $M$ positive eigenvalues. On the other hand, all the eigenvalues of $\hRb_M^{1/2}\Rb_M^{-1/2}$ are between its smallest and largest singular values, that is,
\begin{align*}\sqrt{\lmin(\Rb_M^{-1/2}\hRb_M\Rb_M^{-1/2})} &\le \lmin(\hRb_M^{1/2}\Rb_M^{-1/2}) \\ 
&\le \lmax(\hRb_M^{1/2}\Rb_M^{-1/2})\le \sqrt{\lmax(\Rb_M^{-1/2}\hRb_M\Rb_M^{-1/2})}.\end{align*}
From Theorem~\ref{th:main_th}, if $\xi_N$ is bounded away from $0$, as $N,M\to\infty$ with $N/M\to c\in\positR$, we have almost surely
$$\left|\sqrt{\lmax(\Rb_M^{-1/2}\hRb_M\Rb_M^{-1/2})}-\sqrt{\xi_N}\right|=O(M^{-1/2+\varepsilon}).$$
The same result also holds for $\sqrt{\lmin(\Rb_M^{-1/2}\hRb_M\Rb_M^{-1/2})}$. Therefore, we have almost surely
$$\left\|\Rb_M^{-1/4}\hRb_M^{1/2}\Rb_M^{-1/4}-\sqrt{\xi_N}\Ib\right\|=O(M^{-1/2+\varepsilon}).$$
Taking the above estimations into \eqref{eq:corol_1/2conv}, the result follows.

\subsection{Proof of Proposition~\ref{prop:lsd_consistent}}\label{subsec:proof_lsd_consist}
We first recall the LSD of Toeplitz matrices. If a sequence of Toeplitz matrices $(\Rb_M=(r_{i-j})_{i,j=1}^M)_{M\ge 1}$ have a real spectral density $f\in L^1(-\pi,\pi)$, then by a generalized version of Szeg\H o's Theorem \cite[Theorem~2]{capizzano2002test}, for any continuous function $\varphi$ defined on $\R$ such that $\varphi(x)/(1+|x|)$ is bounded, we have
\begin{equation}
    \lim_{M\to\infty}\frac{1}{M}\sum_{k=1}^M\varphi(\lambda_k(\Rb_M))=\frac{1}{2\pi}\int_{-\pi}^{\pi}\varphi(f(\theta))\dd\theta. \label{eq:szego_cv}
\end{equation}
In particular, the LSD $\mu^{\Rb}$ of $\Rb_M$ will be defined by the identity 
\begin{equation}
\int\varphi\dd \mu^{\Rb}=\frac{1}{2\pi}\int_{-\pi}^{\pi}\varphi(f(\theta))\dd \theta\,, \quad \forall \varphi \in C_b(\R)\,, \label{eq:def_mu_Rb}
\end{equation}
where $C_b(\R)$ denotes the set of bounded continuous functions on $\R$. 

In order to prove Proposition~\ref{prop:lsd_consistent}, we are led to a general result relating the ratio ESD of two Toeplitz matrices with their spectral densities, which may be of independent interest. 

\begin{lemma}\label{lem:ratio_lsd_cv_toep} Let $(\Rb_M=(r_{i-j})_{i,j=1}^M)_{M\ge 1}$ be a sequence of Toeplitz matrices with real spectral density $f\in L^1(-\pi,\pi)$. Let $(f_M)_{M\ge 1}$ be a sequence of real functions in $L^1(-\pi,\pi)$. Let $\Rb_M^{(M)}=(r^{(M)}_{i-j})_{i,j=1}^M$ with
$$r^{(M)}_{k}=\frac{1}{2\pi}\int_{-\pi}^{\pi}f_M(\theta)e^{\im k\theta}\dd\theta$$
the Fourier coefficients of $f_M$. Let $\mu^{\Rb}$ denote the LSD of $\Rb_M$ defined in \eqref{eq:def_mu_Rb}. 
\begin{enumerate}
    \item \label{item:f_M_to_f} If $\|f_M-f\|_1\to 0$, then 
    $$\mu^{\Rb_M^{(M)}}\cvweak \mu^{\Rb}.$$
    \item In addition to (\ref{item:f_M_to_f}), if moreover $f$ is positive and bounded away from $0$, and $f_M/f>a$ for some $a\in\R$, then
    $$\mu^{\Rb_M^{(M)}\Rb_M^{-1}}\cvweak \delta_1.$$
\end{enumerate}
\end{lemma}

First note that by normalization, we can assume $\xi_N=1$ for all $N\in \N$ without loss of generality. Then from the proof of Proposition~\ref{prop:inconsis_Rb}, almost surely, as $M,N\to\infty$ with $N\gg \log^3 M$,
$$\sup_{\theta}\left|\frac{\hUpsilon_M^b(\theta)}{\Upsilon_M^b(\theta)}-1\right|\to 0.$$
Then almost surely
$$\int_0^{2\pi}|\hUpsilon_M^b(\theta)-\Upsilon_M^b(\theta)|\dd\theta\le \sup_{\theta}\left|\frac{\hUpsilon_M^b(\theta)}{\Upsilon_M^b(\theta)}-1\right|\int_0^{2\pi} |\Upsilon_M^b(\theta)|\dd\theta \to 0,$$
where $\int_0^{2\pi} |\Upsilon_M^b(\theta)|\dd\theta$ is bounded because $\Upsilon_M^b(\theta)$ is the Ces\`aro mean of the Fourier series of $f$, and it is well known that $\Upsilon_M^b$ converges to $f$ in $L^1(0,2\pi)$. Then we deduce that almost surely 
$$\int_0^{2\pi}|\hUpsilon_M^b(\theta)-f(\theta)|\dd\theta\to 0.$$
Also note that $\hUpsilon_M^b/f\ge 0$, then $\hUpsilon_M^b$ and $f$ satisfy the conditions of Lemma~\ref{lem:ratio_lsd_cv_toep}. Therefore the result of Proposition~\ref{prop:lsd_consistent} is a corollary of Lemma~\ref{lem:ratio_lsd_cv_toep}. 

It remains to prove Lemma~\ref{lem:ratio_lsd_cv_toep}. If $f_M$ converges in $L^1(0,2\pi)$ to $f$, we denote
$$(f_M-f)_+ = \max\{f_M-f, 0\}, \quad (f_M-f)_- = \max\{ f-f_M, 0\},$$
and
$$\begin{aligned}\Ab_+  &= \left(\frac{1}{2\pi}\int_0^{2\pi}(f_M-f)_+(\theta)e^{\im(i-j)\theta}\dd\theta\right)_{i,j=1}^M, \\ \Ab_- &= \left(\frac{1}{2\pi}\int_0^{2\pi}(f_M-f)_-(\theta)e^{\im(i-j)\theta}\dd\theta\right)_{i,j=1}^M.\end{aligned}$$
Then $\Ab_+,\Ab_-$ are two positive semi-definite Toeplitz matrices satisfying
$$\frac{1}{M}\tr\Ab_\pm=\frac{1}{2\pi}\int_0^{2\pi}(f_M-f)_\pm\dd\theta\to 0.$$
It is easy to prove that there exists a sequence of positive numbers $(\varepsilon_M)_{M\ge 1}$ converging to $0$, such that
$$\frac{\# \{k\tq \lambda_k(\Ab_\pm)>\varepsilon_M\}}{M}\le \varepsilon_M$$
where "$\# S$" denotes the cardinal of the set $S$. From \eqref{eq:szego_cv} we already have $\mu^{\Rb_M}\cvweak \mu^{\Rb}$. We next prove successively that 
\[ \mu^{(\Rb_M+\Ab_+)}\cvweak \mu^{\Rb} \quad\text{ and}\quad  \mu^{(\Rb_M+\Ab_+-\Ab_-)}\cvweak \mu^{\Rb},\]
which is the first result of the lemma. 

Let $\Ab_+=U\diag(\lambda_1,\dots,\lambda_M)U^*$ be a diagonalization of $\Ab_+$ with $\lambda_1,\dots,\lambda_M$ its eigenvalues. Let $\Ab_+^{(1)}=U\diag(\lambda_1\ind_{\lambda_1>\varepsilon_M},\dots,\lambda_M\ind_{\lambda_M>\varepsilon_M})U^*$, and $\Ab_+^{(2)}=U\diag(\lambda_1\ind_{\lambda_1\le\varepsilon_M},\dots,\lambda_M\ind_{\lambda_M\le \varepsilon_M})U^*$. Then the rank of $\Ab_+^{(1)}$ is at most $M\varepsilon_M$, and $\|\Ab_+^{(2)}\|\le \varepsilon_M$. Using Theorem~A.43 and A.45 in \cite{bai2010spectral} successively, we can prove that
$$\mu^{\Rb_M+\Ab_+}\cvweak \mu^{\Rb}.$$
Repeating the same procedure, we have also
$$\mu^{\Rb_M+\Ab_+-\Ab_-}\cvweak \mu^{\Rb}.$$
Thus the first part of the lemma is proved. 

Next we prove the second part. We will first prove the following lemma.

\begin{lemma}\label{lem:my_potential_th}Suppose that the probability measures $\mu_n (n=1,2,\dots)$ and $\mu$ are supported on $[a,+\infty)$. If for any $x<a$, 
\begin{equation}
    \lim_{n\to\infty}\int \log |x-t|\dd\mu_n(t)=\int \log |x-t|\dd\mu(t)<\infty, \label{eq:conv_real_potential}
\end{equation}
then $\mu_n$ converges weakly to $\mu$. 
\end{lemma}
\begin{proof}\let\qed\relax From every subsequence of $(\mu_n)$ we can extract a subsequence converging vaguely to a positive measure $\nu$ with total mass less than or equal to $1$. Take an arbitrary $x_0<a$. Then for any $z\in \C\backslash (x_0,+\infty)$, because the function $t\mapsto (z-t)^{-1}$ is continuous on the support of $\mu_n$ and $\nu$, and tends to $0$ as $t\to\infty$, we have the convergence of Stieltjes transform
$$s_n(z):=\int\frac{1}{z-t}\dd\mu_n(t) \xrightarrow[n\to\infty]{}\int\frac{1}{z-t}\dd\nu(t)=:s(z).$$
By dominated convergence theorem, we have
$$\int_{x_0}^z s_n(w)\dd w\xrightarrow[n\to\infty]{} \int_{x_0}^z s(w)\dd w,$$
where the integral is taken along the segment from $x_0$ to $z$. Changing the order of integrals, we get
$$\int(\log(z-t)-\log(x_0-t))\dd\mu_n(t)\xrightarrow[n\to\infty]{}\int_{x_0}^z s(w)\dd w,$$
where $\log z=\log |z|+i\arg z$ with $\arg z\in [0,2\pi)$. When $z=x\in (-\infty,x_0]$, the above convergence and the condition \eqref{eq:conv_real_potential} imply that
$$\int_{x_0}^x s(w)\dd w=\int\log(x-t)\dd\mu(t)-\int\log(x_0-t)\dd\mu(t).$$
Extending this equality by analyticity, we have, for $z\in \C\backslash (x_0,+\infty)$,
$$\int_{x_0}^z s(w)\dd w= \int\log(z-t)\dd\mu(t)-\int\log(x_0-t)\dd\mu(t).$$
Differentiating both sides, we get
$$s(z)= \int\frac{1}{z-t}\dd\mu(t).$$
This implies that $\mu=\nu$. Then because the vague limit $\mu$ is a probability measure, we actually have the weak convergence $\mu_n\cvweak \mu$ and Lemma~\ref{lem:my_potential_th} is proved.
\end{proof}

We continue the proof of Lemma~\ref{lem:ratio_lsd_cv_toep}. By Lemma~\ref{lem:my_potential_th}, we only need to prove \begin{equation}
    \lim_{M\to\infty}\frac{1}{M}\log \left|\det(\Rb^{(M)}_M-x\Rb_M)\right|=\frac{1}{2\pi}\int_0^{2\pi}\log|f(\theta)-xf(\theta)|\dd\theta. \label{eq:limit_potential_mov}
\end{equation}
for $x\in (-\infty,a)$. From the condition $f_M/f>a, f>0$, we get $f_M-xf>0$ for $x<a$, thus the Toeplitz matrix $\Rb^{(M)}_M-x\Rb_M$ is positive definite. For any $A>1$ and $t>0$, let 
$$\ell_A(t):=\min(\log(t), A), h_A(t):=\max(\log(t),A)-A.$$
Then it is easily seen that $\ell_A(t)+h_A(t)=\log(t)$, and $\ell_A(t)\to\log (t)$, $h_A(t)\to 0$ for every $t>0$ when $A\to\infty$. Note that $f_M-xf\to (1-x)f$ in $L^1$, then from the first part of the lemma, we have
\begin{equation}
    \frac{1}{M}\sum_{k=1}^M\ell_A(\lambda_k(\Rb^{(M)}_M-x\Rb_M))\xrightarrow[M\to\infty]{} \frac{1}{2\pi}\int_0^{2\pi} \ell((1-x)f(\theta))\dd \theta. \label{eq:sum1_josf}
\end{equation}
Note that if $A$ is large, $h_A(t)<\sqrt{t}$, thus
\begin{equation}
    \begin{aligned}\frac{1}{M}\sum_{k=1}^M h_A(\lambda_k(\Rb^{(M)}_M-x\Rb_M)) & \le \frac{1}{M}\sum_{k=1}^M (\lambda_k(\Rb^{(M)}_M-x\Rb_M))^{\frac{1}{2}}\ind_{\lambda_k(\Rb^{(M)}_M-x\Rb_M)>A} \\
& \le \frac{1}{M\sqrt{A}}\tr (\Rb^{(M)}_M-x\Rb_M) \\
& = \frac{1}{2\pi\sqrt{A}}\int_0^{2\pi} (f_M-xf)(\theta)\dd\theta \end{aligned} \label{eq:sum2_jost}
\end{equation}
Summing \eqref{eq:sum1_josf} and \eqref{eq:sum2_jost}, and let $A\to+\infty$, we get
\begin{equation}
    \lim_{M\to\infty}\frac{1}{M}\log \left|\det(\Rb^{(M)}_M-x\Rb_M)\right|=\frac{1}{2\pi}\int_0^{2\pi}\log|f(\theta)-xf(\theta)|\dd\theta. \label{eq:log_det_toep_conv_1}
\end{equation}

On the other hand, by \cite[Theorem~2]{capizzano2002test} and the hypothesis that $f$ is positive and bounded away from $0$,  we have
\begin{equation}
    \lim_{M\to\infty}\frac{1}{M}\log \left|\det(\Rb_M)\right|=\frac{1}{2\pi}\int_0^{2\pi}\log|f(\theta)|\dd\theta. \label{eq:log_det_toep_conv}
\end{equation}
Take the difference of \eqref{eq:log_det_toep_conv_1} and \eqref{eq:log_det_toep_conv}, we get
$$\frac{1}{M}\log \left|\det(\Rb_M^{-1}\Rb^{(M)}_M-x\Ib)\right|\xrightarrow[M\to\infty]{} \log |1-x|$$
for $x<a$. From Lemma~\ref{lem:my_potential_th}, we have $\mu^{\Rb_M^{-1}\Rb_M^{(M)}}\cvweak \delta_1$. The proof of Lemma~\ref{lem:ratio_lsd_cv_toep} is complete.

\subsection{Proof of Proposition~\ref{prop:consistency_pca_whitened}}\label{subsec:proof_pca_whitened}
By Proposition~\ref{prop:whitening_consistence}, we have
$$\|\Sb_w-\xi_N^{-1}\Sb_Y\|\xrightarrow[]{\as}0.$$
Let $\vb_1,\dots,\vb_p$ and $\vb_{w,1},\dots,\vb_{w,p}$ be the eigenvectors associated with the $p$ largest eigenvalues of $\Sb_Y$ and $\Sb_w$, respectively. Let $\mathbf{P}$ and $\mathbf{P}_w$ be the orthogonal projections into the subspace generated by $\{\vb_1,\dots,\vb_p\}$ and $\{\vb_{w,1},\dots,\vb_{w,p}\}$, respectively. Using the integral formula for eigenprojections, we have
$$\mathbf{P}_w-\mathbf{P}= \frac{1}{2\pi\im}\int_\Gamma [(z-\Sb_w)^{-1}-(z-\xi_N^{-1}\Sb_Y)^{-1}]\dd z,$$
where $\Gamma$ is a contour on the complex plane surrounding the $p$ largest eigenvalues of $\Sb_w$ and $\xi_N^{-1}\Sb_Y$, and keeping the other eigenvalues outside. Because there is a gap between bulk and spike eigenvalues, the distance between the contour and the spectrum of both matrices can be lower bounded. Therefore the resolvents $(z-\Sb_w)^{-1}$ and $(z-\xi_N^{-1}\Sb_Y)^{-1}$ are both uniformly bounded in $M,N$ and $z\in\Gamma$. Using the formula $A^{-1}-B^{-1}=A^{-1}(B-A)B^{-1}$, we deduce that
$$\|\mathbf{P}_w-\mathbf{P}\|\le \frac{1}{2\pi}\int_\Gamma \|(z-\Sb_w)^{-1}\| \|\Sb_w-\xi_N^{-1}\Sb_Y\|\|(z-\xi_N^{-1}\Sb_Y)^{-1}\| |\dd z|\xrightarrow[]{\as}0.$$

On the other hand, by \cite[Theorem~11.3]{yao2015sample}, we know that the matrix $M^{-1/2}\Yb$ is almost surely bounded in spectral norm. Together with Corollary~\ref{corol:conv_1/2power}, we have
$$\frac{1}{\sqrt{M}}\|\Yb_w-\Yb/\sqrt{\xi_N}\|\le \left\|\frac{1}{\sqrt{M}}\Yb\right\|\|\Rb_M^{1/2}\hRb_M^{-1/2}-\xi_N^{-1/2}\|\xrightarrow[]{\as}0.$$
Therefore, 
$$\frac{1}{\sqrt{M}}\|\Yb_{w,pc}-\Yb_{pc}/\sqrt{\xi_N}\|\le \frac{1}{\sqrt{M}}\|\mathbf{P}_w-\mathbf{P}\|\|\Yb_w\|+ \frac{1}{\sqrt{M}}\|\mathbf{P}\|\|\Yb_w-\Yb/\sqrt{\xi_N}\|\xrightarrow[]{\as}0.$$
Then we automatically have 
$$\frac{1}{\sqrt{M}}\|\Yb_{w,pc}-\Yb_{pc}/\sqrt{\xi_N}\|_F\xrightarrow[]{\as}0,$$
because 
$\mathrm{rank}(\Yb_{w,pc}-\Yb_{pc}/\sqrt{\xi_N})\le \mathrm{rank}(\Yb_{w,pc})+\mathrm{rank}(\Yb_{pc})=2p$.

\bibliography{main.bib}{}
\bibliographystyle{plain}
\end{document}